\documentclass[preprint,11pt]{elsarticle}

\journal{Journal of Theoretical Biology}
\usepackage{fullpage}
\usepackage{float}
\restylefloat{figure}
\usepackage{amsmath,amsthm,amssymb}

\usepackage{graphicx}

\usepackage{amssymb,graphicx,float,epsf}
\def\RR{\hbox{I\kern-.2em\hbox{R}}}

\usepackage{hyperref}
\usepackage{cleveref}
\usepackage{amsfonts}

\newtheorem{Lemma}{Lemma}

\newtheorem{Prop}{Proposition}
\newtheorem{Rem}{Remark}
\newtheorem{Ex}{Example}
\newtheorem{Th}{Theorem}

\begin{document}

\begin{frontmatter}

\title{On 
the interplay of harvesting 
and various diffusion strategies for spatially heterogeneous populations
} 

\author[label1]{Elena Braverman}
\author[label2]{Ilia Ilmer}

\address[label1]{Dept. of Math. and Stats., University of
Calgary,2500 University Drive N.W., Calgary, AB, Canada T2N 1N4; e-mail
maelena@ucalgary.ca, phone 1-(403)-220-3956, fax 1-(403)--282-5150 (corresponding author)}

\address[label2]{Dept. of Math. and Stats., University of
Calgary and Dept. of Computer Science, City University of New York-College of Staten Island, 
The Graduate Center of City University of New York, 365 Fifth Avenue, New York 10314}





\begin{abstract}
The paper explores the influence of harvesting (or culling) on the outcome of the competition
of  two species in a spatially heterogeneous environment. The harvesting effort is assumed to be proportional to the space-dependent intrinsic growth rate.
The differences between the two populations are the diffusion strategy and the harvesting intensity. 
In the absence of harvesting, competing populations may either coexist, or one of them may bring the other to extinction. If the latter is the case, introduction of any level of harvesting to the successful species guarantees survival to its non-harvested competitor. In the former case, there is a strip of ``close enough" to each other harvesting rates leading to preservation of the original coexistence.
Some estimates are obtained for the relation of the harvesting levels providing either coexistence or competitive exclusion. 
\end{abstract}


\end{frontmatter}

\section{Introduction}

Harvesting is quite a common cause for extinction of exploited populations. However, the study
of one harvested species \cite{Roques2007}, generally, does not adequately describe the population dynamics in multi-species environment, as it ignores interactions, such as competition, predation or mutualism. If, in the absence of harvesting, two populations coexist, will intervention bring the harvested population to extinction? If originally the resident species was immune to a possible invasion
by a competitor, will its being harvested open the habitat to a possible invasion? If two species coexist and compete for the same resources, and both are harvested, which relations between the harvesting rates will preserve coexistence?   

Spatial structure, as well as competition for resources with other species, can significantly influence the outcome of the competition. Structural or spatial heterogeneity can be beneficial for exploited species, for example, providing refuge areas, or harmful, for instance, creating age-specific over-exploitation, or crucial disruptions in the spread over habitat, or vulnerability under seasonal migrations.  

Studies on harvesting were usually focused on establishing the maximum sustainable yield, investigating at the same time optimality and sustainability of harvesting policies. Usually, they were limited to one population, see, for example, 
\cite{Ali2009,Bai2005,Brav,Engen2014,Goddard2011,Korobenko2013,Roques2007}. 
Special caution is required for the choice of harvesting strategies when the Allee effect is incorporated in population dynamics \cite{Ali2009,Goddard2011}.
More interesting scenarios arise when harvesting is applied to one or more of several interacting populations 
\cite{Chadhouri1986,Clayton1997,Ekerhovd2016,Engen2014,LiuBai2016,Rowcliffe2003,Zhou1982}. 
In certain cases, a generalist predator was introduced to control invasion 
and was considered as a source of harvesting \cite{Madec2017}. 

While either homogeneous environments (modeled by a system of ordinary differential equations)  or non-selective harvesting can result in competitive exclusion of all but one species \cite{Chadhouri1986,Engen2014}, consideration of structured environment and different types of harvesting applied to distinct species, can lead to coexistence as well, see, for example, \cite{Zhou1982}. 
In certain cases, populations with and without diffusion  \cite{Madec2017}, or in the case of small diffusion coefficients \cite{Ang2016} were considered.
In the non-harvested case for two competing species, 
common scenarios are either competitive exclusion or coexistence \cite{Cantrell_Cosner2017}.
If competitors differ by one parameter only, for example, dispersal strategies or carrying capacities, this led to a competitive exclusion to favor a more evolutionarily advantageous property or strategy.
Still, for harvested species, the interplay of two or more factors (different spatial distributions, model parameters and harvesting rates) can eventually lead to anyone of the three different scenarios, two of competitive exclusion, and coexistence. In principle, if both species are over-exploited, the fourth scenario of the total extinction is possible.

In the present paper, we investigate a competition of two species choosing different dispersal strategies and being harvested with various harvesting efforts.
Evolutionary advantage or disadvantage of fast and slow, random or directed movements are closely connected to heterogeneity of the environment.  
For spatially homogeneous environments and regular dispersal, a slower diffuser \cite{Doc} is a competition winner, if no other differences between competitors exist. If the environment is spatially heterogeneous then the species using knowledge on the carrying capacity of the environment in its dispersal strategy can bring its regularly diffusing competitor to extinction
\cite{BKK2015,C4,C3,Korobenko2014}.
The ideal free distribution describes the spatial structure when any movement would decrease the population fitness. 
For spatially heterogeneous but stationary environments, this may correspond to the case of the population density  
coinciding with the carrying capacity of the environment. 
The idea that an evolutionarily stable strategy corresponds to the case, when the ideal free distribution is a limit solution was justified, for example, in \cite{Averill,C1}. However, the logistic equation with a regular diffusion (the Fisher equation) does not have an ideal free distribution as a solution, which is extensively discussed in the recent paper \cite{Ang2016}. Moreover, faster diffusion leads to a disadvantage in a competition \cite{Doc}.
Several developments, such as introducing advection or considering the modified dispersal term, were aimed to remedy this situation \cite{Averill,Brav,C1,C4,C2,C3,Lam2014}. An interesting investigation in this area in the discrete case  was recently reported in \cite{Cantrell_Cosner2017}. In the continuous case for two (or more) species, cooperation in resources consumption is possible when the densities of the two populations complement each other forming an ideal free pair \cite{Averill,CAMWA2016}. For such a pair, the total density exactly matches the carrying capacity, and the diffusion strategies are aligned with these distributions. This can describe the situation when accessibility of resources is spatially different for the two species, i.e. specialization by resource consumption. 

For a given type of diffusion, the average population rates 
depend on the two important parameters, assumed to be space-dependent in heterogeneous environments: the intrinsic growth rate $r$ and the carrying capacity $K$, see \cite{Ang2016,Engen2014} and references therein.
If a population is exploited, different harvesting strategies can be introduced, such as constant level, proportional (constant effort) and threshold harvesting. 
For the latter case, the deduction is constant for high enough population rates but either becomes close to proportional 
\cite{Roques2007} or is canceled \cite{Engen2014} when the population level is below a prescribed threshold. 

Most results on the outcomes of the competition stipulated by various diffusion strategies and other parameters, are based on the theory of monotone dynamical systems, see \cite{Hsu} and its recent generalization in \cite{Lam2016}, and the eigenvalue technique developed and used in \cite{CC,C1,C4,C2,C3}.

In several cases, models with harvesting can be reduced to non-harvested systems with modified parameters, see 
\cite{Brauer1982}. We apply a similar modification which leads to the update of the original space-dependent intrinsic growth rates and carrying capacities.
The main focus of the present paper is on the interplay of various diffusion strategies and levels of harvesting for two competing populations. To efficiently handle the model, we introduce the following assumptions:

\begin{itemize}
\item
both competing species have proportional time-independent intrinsic growth rates and the same stationary carrying capacities; 
\item
one or both species can be subject to space-dependent harvesting, and the spatial harvesting level is proportional to the intrinsic growth rate;
\item
diffusion strategies can vary but both populations are isolated, there is no flux through the boundary of the common habitat.
\end{itemize}

The first assumption that intrinsic growth rates and carrying capacities coincide is quite common for a comparison, whenever 
one or more other properties are evaluated (in our case, a diffusion strategy and a harvesting level). 
Harvesting of a percentage of new growth, which is proportional to the current population and the space-dependent intrinsic growth rate, is quite natural in resource management. This means that higher harvesting efforts are applied to more fertile areas with higher harvest expectations. The third assumption that both populations are considered in an isolated area is also general in population dynamics. Evolutionary advantage is evaluated without the interference of emigration and immigration through the boundaries. The system describes an isolated area and can be modeled in laboratory conditions.

Under certain assumptions, an equation describing a harvested species can be reduced to a regular diffusive logistic
equation, with an intrinsic growth rate and carrying capacity updated. 
Combined with non-harvested species, we obtain a system with different diffusion strategies 
and carrying capacities considered earlier, and distinct intrinsic growth rates, its particular cases were investigated in 
\cite{BKK2015}. 

Following \cite{BKK2015}, we can conclude that under any level of harvesting, the advantage of 
the carrying capacity driven dispersal is no longer sufficient to bring the competing species to extinction and observe,
using numerical simulations, how a growing harvesting rate brings exploited population to extinction.
However, in the present paper we in addition evaluate harvesting levels for which coexistence 
and extinction should happen. In particular, explicit sufficient estimates are given for each scenario to occur,  
unlike \cite{BKK2015}, where observations on transitions from coexistence to extinction were based 
on numerical simulations only. 
Moreover, in the present paper, sufficient coexistence bounds are compared to numerically computed maximal harvesting levels when species coexist. 

We also consider the situation when, without harvesting, the two populations sustain. 
This is attained when the dispersal strategies form an ideal free pair \cite{Averill,CAMWA2016} and
can describe some specialization (for example, one species foraging primarily in shallow and the other in deep water).
In this case, both species can be subject to harvesting. 
If harvesting efforts are close, there is coexistence. If we fix a harvesting effort for the second species, as harvesting of the first one exceeds this effort and increases, its average density  decreases, necessarily bringing it to extinction for intensive enough harvesting. However, for this fixed effort being small enough, lower levels of harvesting
and even its absence cannot, generally, bring the second species to extinction. 
Some estimates when each of the the three scenarios necessarily happens, are presented. 

The paper is organized as follows. In Section~\ref{sec:prelim}, we state the problem and the main results of the paper. Section~\ref{sec:numer} illustrates these conclusions with numerical simulations and compares theoretically computed constants to numerical bounds.
In Section~\ref{sec:discuss} we discuss conclusions of the present paper and outline some further directions in the study of the influence of dispersal strategies on extinction and survival of exploited populations.
Section~\ref{sec:proofs} contains proofs of the main results and all the auxiliary statements required.

\section{Preliminaries and Main Results}
\label{sec:prelim}

We consider an isolated spatially distributed competitive system, with harvesting 
effort proportional to the spatially heterogeneous but time-independent intrinsic growth rate. 
Both populations exist in a spatial domain, in an isolated environment, there is no flux through the boundary of the domain. 

Let $\Omega$ be an open nonempty bounded domain of $\mathbb{R}^n$ with 
$\partial\Omega\in C^{2+p}$, $0<p<1$.
We assume that $K(x)$, $K_i(x)$, $i=1,2$, $P(x)$ and $Q(x)$  are in the class $C^{1+p}(\overline{\Omega})$, and they are positive for any $x \in \overline{\Omega}$, $r_i>0$, $i=1,2$, $r(x) \geq 0$, $x \in \overline{\Omega}$, and $r(x) >0$ in an open nonempty subdomain of $\Omega$.
Denote by $u(t,x)$ and $v(t,x)$ population densities, assume that the carrying capacity  $K(x)$ and the intrinsic growth rate $r(x)$ are time independent.

The main object of the paper is the system
\begin{equation}
\label{eq:main_problem}
\begin{aligned}
&\frac{\partial u(t,x)}{\partial t} = \nabla \cdot \left[ a(x) \nabla \left( \frac{u(t,x)}{P(x)} \right) \right] + r(x)u(t,x) \left(1-\frac{u(t,x)+v(t,x)}{K(x)} \right) -\alpha r(x)u(t,x), \\
&\frac{\partial v(t,x)}{\partial t} = \nabla \cdot \left[ b(x) \nabla \left( \frac{v(t,x)}{Q(x)} \right) \right] + r(x)v(t,x) \left(1-\frac{u(t,x)+v(t,x)}{K(x)} \right)-\beta r(x)v(t,x), \\
&t>0, \quad x \in \Omega, \\ 
&\frac{\displaystyle \partial u}{\displaystyle \partial n} - \frac{\displaystyle u}{\displaystyle P} \frac{\displaystyle 
\partial P}{\displaystyle \partial n}  
=\frac{\displaystyle \partial v}{\displaystyle \partial n}  - 
\frac{\displaystyle v}{\displaystyle Q} \frac{\displaystyle \partial Q}{\displaystyle \partial n} =0,~x\in 
{\partial}{\Omega},\\
& u(0,x)=u_0(x), \;v(0,x)=v_0(x),\;x\in{\Omega}.
\end{aligned}
\end{equation}
For the initial conditions, we assume that $u_0(x) \geq 0$, $v_0(x) \geq 0$, $x \in \overline{\Omega}$, and both $u_0$ and $v_0$ are positive in an 
open nonempty subdomain of $\Omega$.

The type of diffusion in \eqref{eq:main_problem}, which was discussed in \cite{CAMWA2016}, includes the following models as special cases (here we consider the first equation in \eqref{eq:main_problem} only):
\begin{itemize}
\item
If $P$ and $a$ are both constant, we obtain  the regular diffusion $d \Delta u$, $d>0$.
\item
If $a$ is constant, the dispersal term $\Delta(u/P)$ in the particular case when $P= \frac{1}{d} K$   
was introduced in \cite{Brav} and later used in \cite{Korobenko2013,Korobenko2014}, while 
the constant $P$ and $a=d/K$, $d>0$, leading to $d \nabla \cdot (\frac{1}{K}\nabla v)$, 
was studied in \cite{Korobenko2014}.
\item
If $a_1 =\mu_1 P$, $r=K$ and $\displaystyle P =e^{\mu_2 K}$, 
where $\mu_i$, $i=1,2$ are space-independent, we 
obtain the directed advection $\displaystyle \nabla \cdot \left[ \mu \nabla u - \nu u \nabla K \right]$
with $\mu=\mu_1$, $\nu=\mu_1 \mu_2$, considered, for example, in \cite{Averill,C4,C2,C3}.
\end{itemize}

We will also use a modification of \eqref{eq:main_problem} without harvesting ($\alpha=\beta=0$) 
but with different proportional growth rates and carrying capacities
\begin{equation}
\label{eq:main_modified}
\begin{aligned}
&\frac{\partial u(t,x)}{\partial t} = \nabla \cdot \left[ a(x) \nabla \left( \frac{u(t,x)}{P(x)} \right) \right] + r_1 r(x)u(t,x) \left(1-\frac{u(t,x)+v(t,x)}{K_1(x)} \right), \\
&\frac{\partial v(t,x)}{\partial t} = \nabla \cdot \left[ b(x) \nabla \left( \frac{v(t,x)}{Q(x)} \right) \right] + r_2 r(x)v(t,x) \left(1-\frac{u(t,x)+v(t,x)}{K_2(x)} \right), \\
&t>0, \quad x \in \Omega, \\
&\frac{\displaystyle \partial u}{\displaystyle \partial n} - \frac{\displaystyle u}{\displaystyle P} \frac{\displaystyle 
\partial P}{\displaystyle \partial n}  
=\frac{\displaystyle \partial v}{\displaystyle \partial n}  - 
\frac{\displaystyle v}{\displaystyle Q} \frac{\displaystyle \partial Q}{\displaystyle \partial n} =0,~x\in 
{\partial}{\Omega},\\
&u(0,x)=u_0(x), \;v(0,x)=v_0(x),\;x\in{\Omega}.
\end{aligned}
\end{equation}

Further, we explore semi-trivial stationary solutions $(u^*,0)$ and $(0,v^*)$, where $u^*$ and $v^*$ satisfy
\begin{equation}
\nabla \cdot \left[ a(x) \nabla \left( \frac{u^{\ast}(x)}{P(x)} \right) \right] + 
r_1 r(x)u^* (x) \left( 1-\frac{u^*(x)}{K_1(x)} \right)=0,\;x\in\Omega,~~~
\frac{\partial (\frac{u^*}{P})}{\partial n}=0,\;x\in\partial\Omega
\label{semi_u}
\end{equation}
and
\begin{equation}
\nabla \cdot \left[ b(x) \nabla \left( \frac{v^*(x)}{Q(x)} \right) \right]+
r_2 r(x) v^*(x) \left( 1-\frac{v^*(x)}{K_2(x)} \right) =0,\;x\in\Omega,~~~
\frac{\partial (\frac{v^*}{Q})}{\partial n}=0,\;x\in\partial\Omega,
\label{semi_v}
\end{equation}
respectively.

\begin{Lemma} \cite{Korobenko2014}
\label{lem0}
If $P$ is proportional to $K_1$ then $u^*(x)=K_1$ is the only solution of \eqref{semi_u}. Problem \eqref{semi_v}
has a unique solution $v^*$ which is positive on $\overline{\Omega}$.
\end{Lemma}


{\bf 

For \eqref{eq:main_problem}, a solution can either tend to one of the semi-trivial equilibria,
when one of the two species is brought to extinction, or both $u$ and $v$ can persist.
In the persistence case, we will say that $u$ and $v$ strongly persist. 
Numerical simulations illustrate that in this case,
for any initial conditions, the solution converges to a certain coexistence equilibrium. 
Under harvesting of both $u$ and $v$, extinction of both species can occur, but this scenario will be described 
later. 
}

The two main results of the present paper are the following.

\begin{Th}
\label{theorem_harv1}
Let $P$ be proportional to  $K$, while $\displaystyle \nabla \cdot \left[ b(x) \nabla \left( \frac{K}{Q} \right) \right] \not\equiv 0$  on $\Omega$, $\alpha,\beta \in [0,1)$.
Then, the following three scenarios can occur. 
\begin{enumerate}
\item
For any $\beta \in [0,1)$ and $\alpha \leq \beta$, all solutions of \eqref{eq:main_problem} 
converge 
to $((1-\alpha)K,0)$.
\item
For any $\beta<1$, there exists $\alpha_2 \in (\beta,1)$ such that whenever $\alpha\in(\alpha_2,1)$,
all solutions of \eqref{eq:main_problem} converge to $(0,v_{\beta}^*)$, where
$v_{\beta}^*$ satisfies \eqref{semi_v} with $K_2=(1-\beta)K$, $r_2=1-\beta$.
\end{enumerate}
\item
{\bf
For any $\beta \in [0,1)$, 		
there exists $\alpha_1 \in (\beta, 1)$ satisfying 
\begin{equation}
\label{alpha_star}
\alpha_1 \geq \alpha^*= 1- \frac{\int_{\Omega} r v_{\beta}^*~dx} {\int_{\Omega} r K~dx},
\end{equation}
where
$v_{\beta}^*$ is a solution of \eqref{semi_v} with $K_2=(1-\beta)K$, $r_2=1-\beta$,
such that for any 
$\alpha \in (\beta,\alpha_1)$, all solutions of \eqref{eq:main_problem}
strongly persist.
} 
\end{Th}

\begin{Rem}
Theorem~\ref{theorem_harv1} outlines the two stages of the influence of harvesting in the case when a harvested species chooses a better dispersal strategy:
\begin{enumerate}
\item
Any level of harvesting, whatever small it is, alleviates this advantage and first leads to coexistence.
\item
At a certain level of harvesting, depending on the parameters of equations 
(carrying capacity, intrinsic growth rate and dispersal $Q$), 
competitive exclusion of a species being harvested is inevitable.
\end{enumerate} 
If both populations are harvested, a more exploited species cannot be the only survivor. Both the lower and the upper bounds of the critical value for a harvesting effort bringing the influenced species to extinction are evaluated in the process of the proof.  
\end{Rem}

The next result refers to the case when $P,Q$ form {\em an ideal free pair}, i.e. neither $(K,0)$ nor $(0,K)$ 
is a solution of \eqref{eq:main_problem} with $\alpha=\beta=0$ as
\begin{equation}
\label{no_K}
\nabla \cdot \left[ a(x) \nabla \left( \frac{K(x)}{P(x)} \right) \right] \not\equiv 0,~~~
\nabla \cdot \left[ b(x) \nabla \left( \frac{K(x)}{Q(x)} \right) \right] \not\equiv 0, 
\end{equation}
but some positive multiple of $K$ is in the convex hull of $P$ and $Q$, i.e. there exist $\gamma>0$ and $\delta>0$ such that 
\begin{equation}
\label{hull}
K(x)=\gamma P(x)+ \delta Q(x), \quad x \in \Omega.
\end{equation}
Equality \eqref{hull} yields that $(\gamma P, \delta Q)$ is a solution of \eqref{eq:main_problem}.

\begin{Th}
\label{theorem_harv2}
Let $P$, $Q$ and $K$ satisfy \eqref{no_K},  \eqref{hull} hold for some $\gamma>0$, $\delta>0$ and 
$\alpha,\beta \in [0,1)$. 

\begin{enumerate}
\item
{\bf 
For any $\beta \in (0,1)$, 		
there exist $\alpha_1 \in [0,\beta)$ and $\alpha_2\in (\beta, 1)$ such that for any 
$\alpha \in (\alpha_1,\alpha_2)$, all solutions of \eqref{eq:main_problem} strongly persist.
}
\item
For any $\beta \in [0,1)$, there exists $\alpha_3 \in (\beta,1)$ such that whenever $\alpha\in(\alpha_3,1)$,
all solutions of \eqref{eq:main_problem} converge to $(0,v_{\beta}^*)$, where
$v_{\beta}^*$ satisfies \eqref{semi_v} with $K_2=(1-\beta)K$, $r_2=1-\beta$.
\item
There exists $\beta_1 \in (0,1)$ such that for any $\beta \in (\beta_1,1)$ there is an $\alpha_0(\beta) \in [0,\beta)$
such that for any $\alpha \in [0,\alpha_0)$,  all solutions of \eqref{eq:main_problem} converge to $(u_{\alpha}^*,0)$, where
$u_{\alpha}^*$ satisfies \eqref{semi_u} with $K_1=(1-\alpha)K$, $r_1=1-\alpha$.
\end{enumerate}
\end{Th}

\begin{Rem}
If two species form an ideal free pair then harvesting of any of them at a low enough rate does not change the fact of 
coexistence. However, if we fix a harvesting rate for one species high enough, the absence or low harvesting for its competitor can bring it to extinction. On the other hand, if the competitor is harvested intensively enough, the chosen species becomes the only survivor. 
\end{Rem}

\begin{Rem}
{\bf
As in Theorem~\ref{theorem_harv1}, we can present an estimate such that, say, for a fixed $\beta$ and 
\begin{equation}
\label{c_ast_est}
\alpha_2 > \alpha^{*} = 1- \frac{ \int_{\Omega} P r v_{\beta}^*/K~dx}{\int_{\Omega}  r P~dx},
\end{equation}
all solutions of \eqref{eq:main_problem} 
strongly persist.}
\end{Rem}

For completeness,  we state that over-exploitation brings both species to extinction.

\begin{Th}
\label{th_extinction}
If $0 \leq \beta < 1 \leq \alpha$ then all solutions of \eqref{eq:main_problem} with nontrivial in $v$ initial conditions 
converge to $(0,v_{\beta}^*)$, and for $0 \leq \alpha < 1 \leq \beta$, all solutions of \eqref{eq:main_problem} with nontrivial in $u$ initial conditions 
converge to $(u_{\alpha}^*,0)$. If $\alpha \geq 1$ and $\beta \geq 1$, the zero solution is a global attractor.
\end{Th}

Finally, we determine the optimal harvesting policy assuming that we can control dispersal and either one or both of two similar species are appropriate for harvesting. This result generalizes the conclusion for a one-species model in \cite{Brav}. We recall that, assuming an attractive stationary solution $(u_s,v_s)$ of \eqref{eq:main_problem}, we compute Sustainable Yield as
\begin{equation*}
\label{SY}
{\rm SY} := \int_{\Omega} \alpha r(x) u_s(x)~dx.
\end{equation*}
If both species are harvested,
\begin{equation*}
{\rm SY} := \int_{\Omega} \alpha r(x) u_s(x)~dx+ \int_{\Omega} \beta r(x) v_s(x)~dx.
\end{equation*}

\begin{Th}
\label{th_MSY}
Maximum Sustainable Yield (MSY) for \eqref{eq:main_problem} is attained for $P=K$, $\alpha=0.5$, $\beta > 0.5$ and any $Q$ (or $\beta \geq 0.5$ and $\displaystyle \nabla \cdot \left[ b(x) \nabla \left( \frac{K}{Q} \right) \right] \not\equiv 0$), if only 
the first species is harvested (and the second one is subject to culling).
If both species are harvested, there is a variety of strategies $P,Q$ such that $K$ is in their convex hull 
(we choose $P$ and $Q$ such that $P+Q=K$) and $\alpha=\beta=0.5$ leading to MSY.  In both cases,
\begin{equation}
\label{MSY}
{\rm MSY} := \int_{\Omega} \frac{1}{4} r(x) K(x)~dx.
\end{equation}
\end{Th}

\section{Numerical Simulations}
\label{sec:numer}

\begin{Ex}
\label{ex1}
Consider \eqref{eq:main_problem} with $P(x) = K(x) = 2+\cos(\pi x)$, $a(x)=b(x)=Q(x) = 1$ for all $x\in(0,4)$,  $t>0$. 
Let $r(x)\equiv 1.1$,  $\alpha,\beta \in (0,1)$, and the initial conditions be $u_0=2.1=v_0$. 
The system becomes a well-known model of carrying capacity driven diffusion competing with a regularly diffusing population.
We consider the solution at the time $T=2000$ for which it is sufficiently close to a steady state. To illustrate the results 
of 
Theorem~\ref{theorem_harv1}, Parts 2 and 3, we consider solutions at $t=T$ as functions of harvesting effort $\alpha \in [0,1)$ for a fixed $\beta \in [0,1)$. We use the functions $K(x)$, $r(x)$, and the stationary solution to the second equation $v_{\beta}^*(x)$ to compute $\alpha^*$ for each fixed value of the harvesting effort $\beta$ as in \eqref{alpha_star}.
In \autoref{example_1_fig_1} we present a spatial distribution of solutions at $t=T$. Without harvesting, a regularly diffusing 
population goes extinct, while adding harvesting of $u$ at small rates leads to coexistence. 
\begin{figure}[ht]
	\centering
	\includegraphics[width=0.4\linewidth]{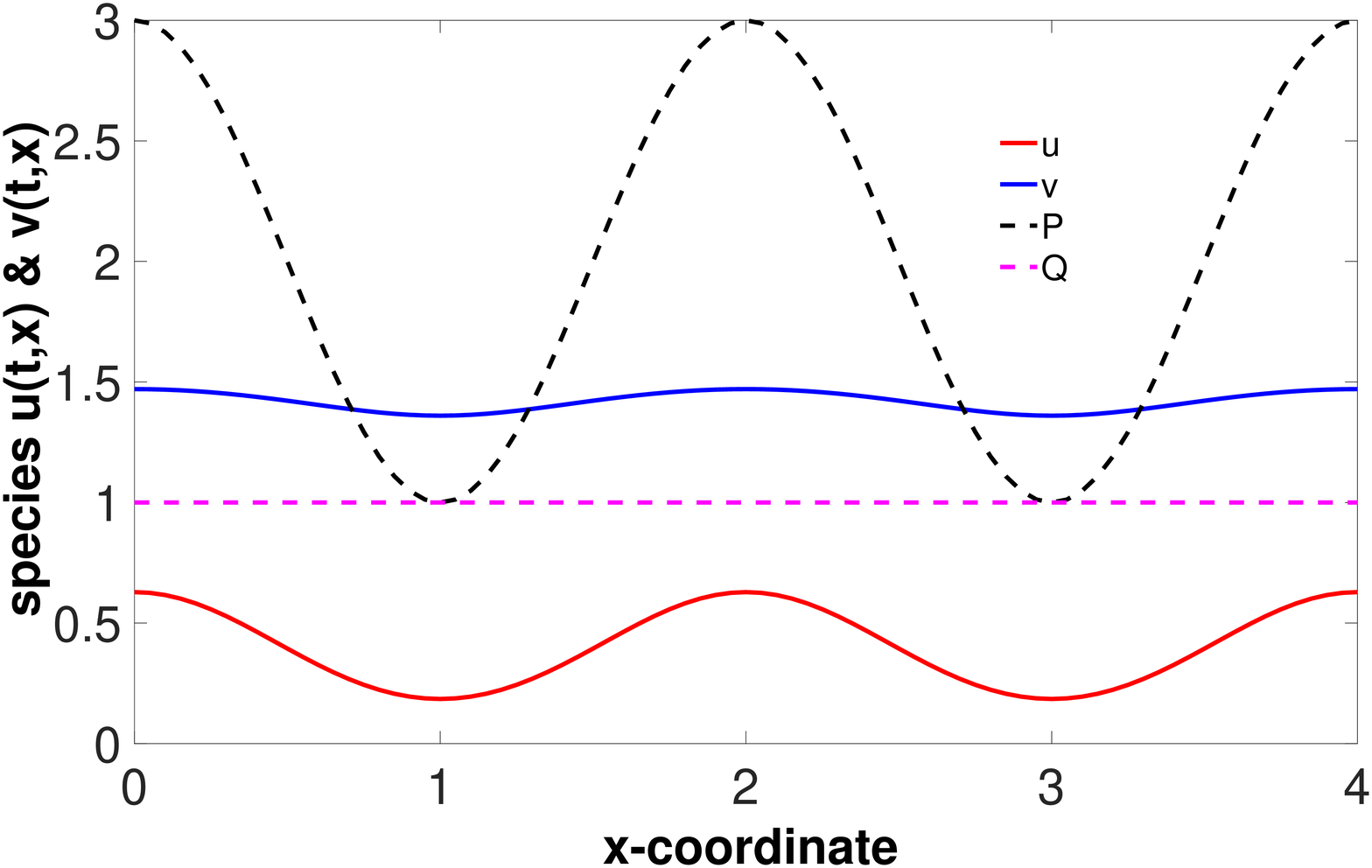}~~
	\includegraphics[width=0.4\linewidth]{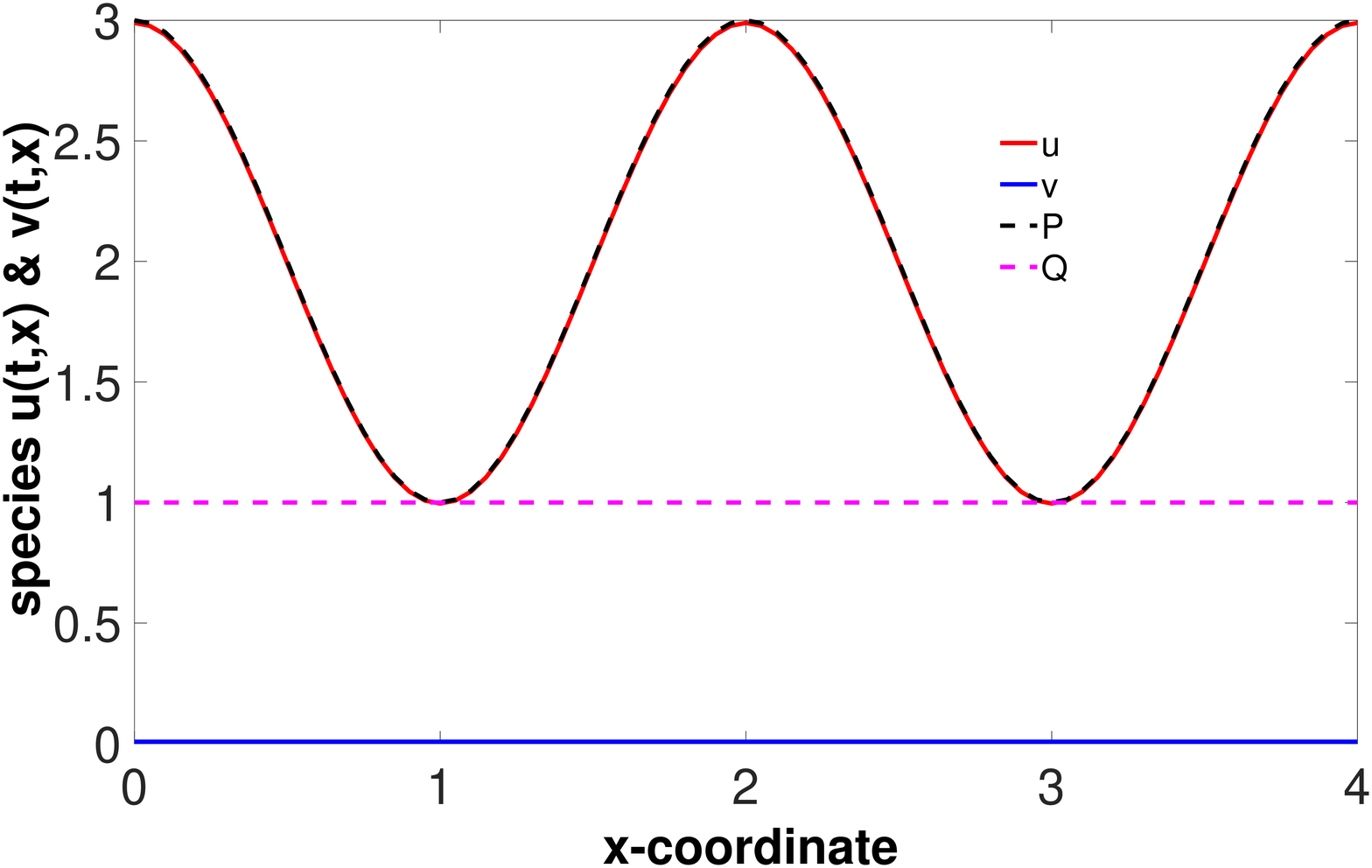}
	\caption{Solutions to \eqref{eq:main_problem} at $t= T=2000$ with $a\equiv b \equiv 1$, $Q\equiv1$, $P = K = 2+\cos(\pi x),~r(x) = 1.1, ~u_0=v_0=2.1, ~x \in (0,4)$ and (left) $\beta=0, \alpha = 0$, (right) $\beta=0, \alpha = 0.1$.} 
	\label{example_1_fig_1}
\end{figure}
\autoref{example_1_fig_2} illustrates the relation between $\alpha^*$ computed in \eqref{alpha_star} for various levels of $\beta$ (which is just a lower estimate of $\alpha$ for which coexistence necessarily occurs) and the maximal
value $\alpha^{**}$ of $\alpha_1$ when coexistence still occurs computed numerically.
\begin{figure}[ht]
\centering
\includegraphics[width=0.4\linewidth]{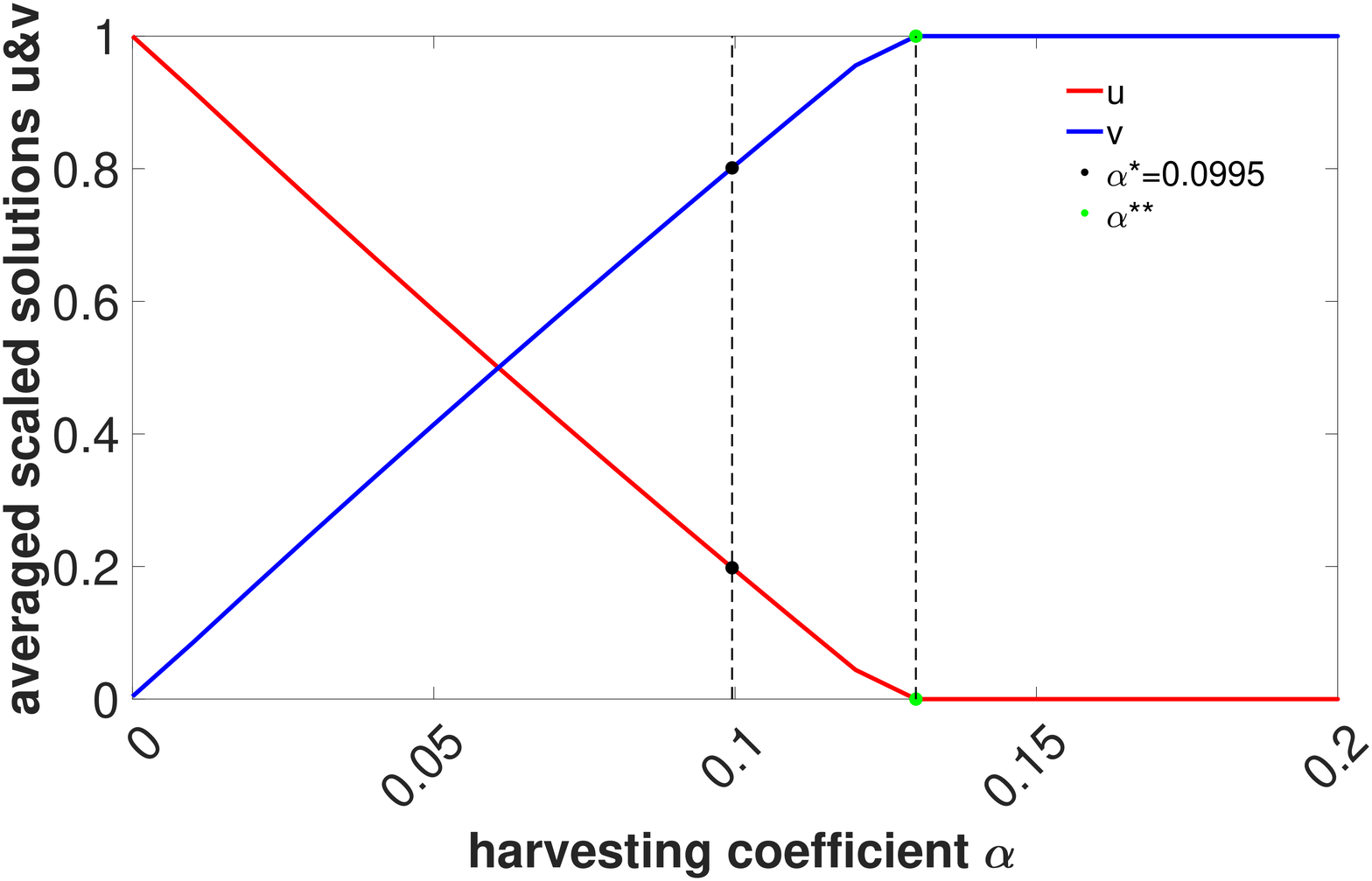}~~
\includegraphics[width=0.4\linewidth]{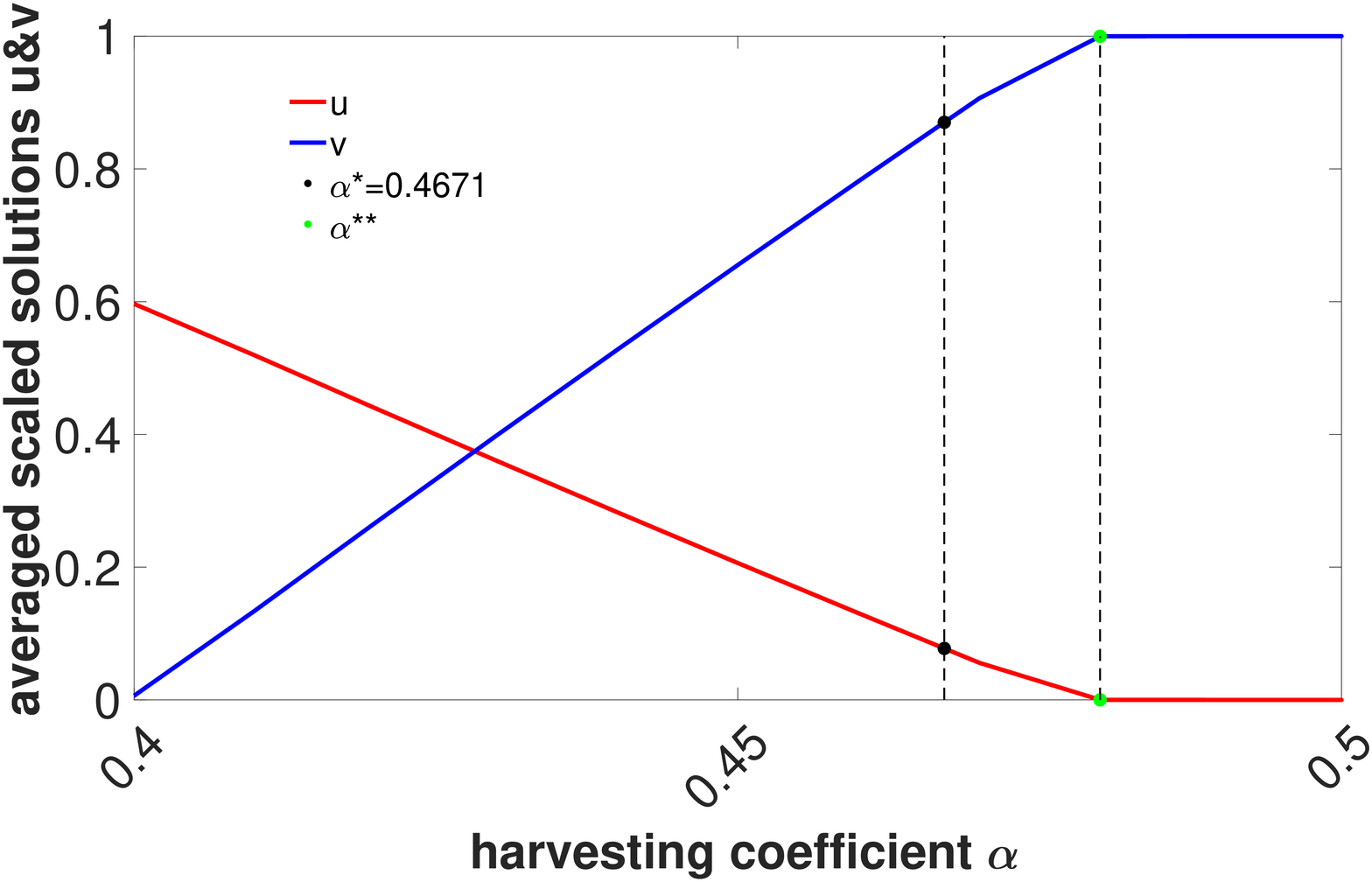}
\includegraphics[width=0.4\linewidth]{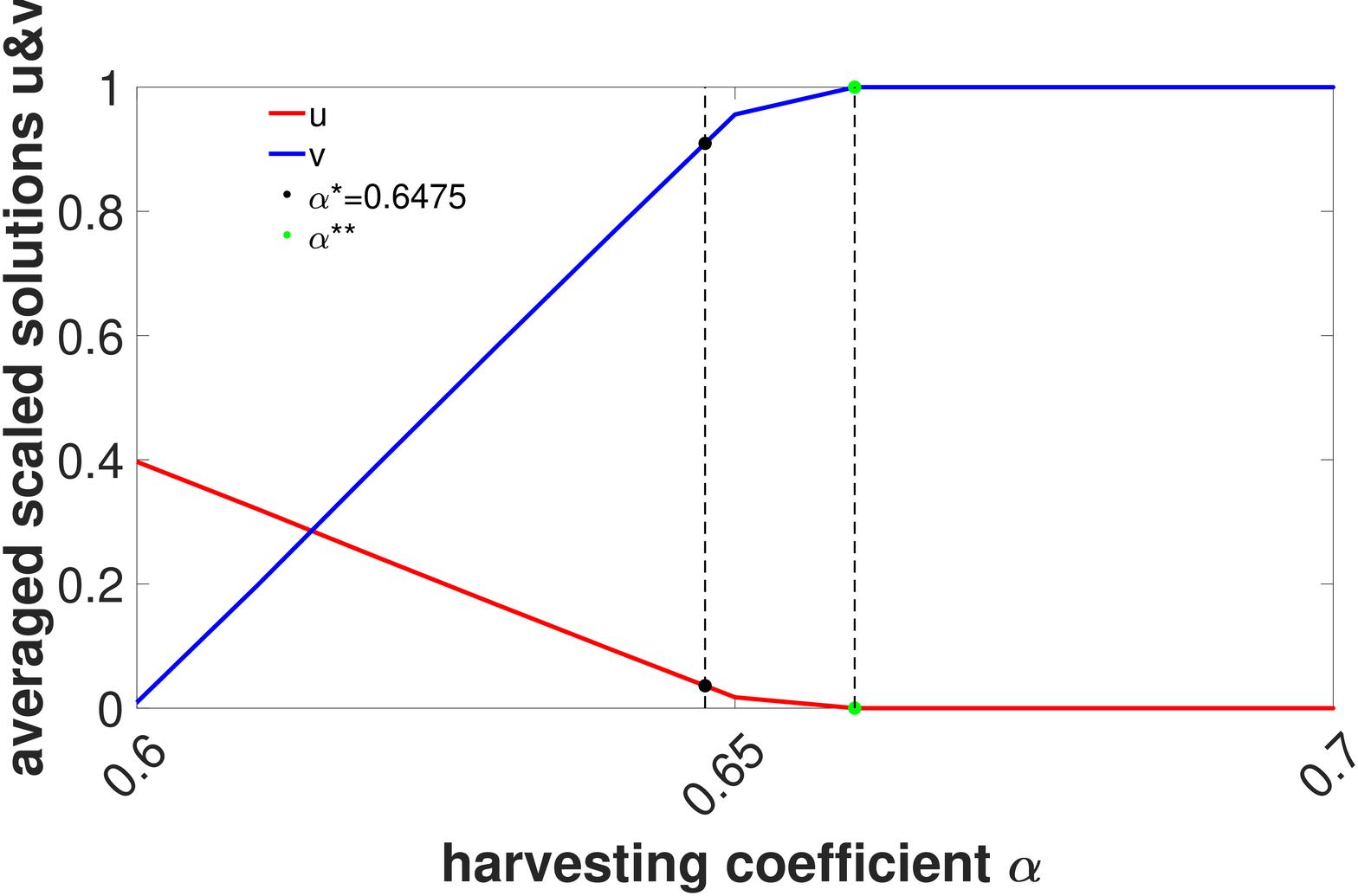}~~
\includegraphics[width=0.4\linewidth]{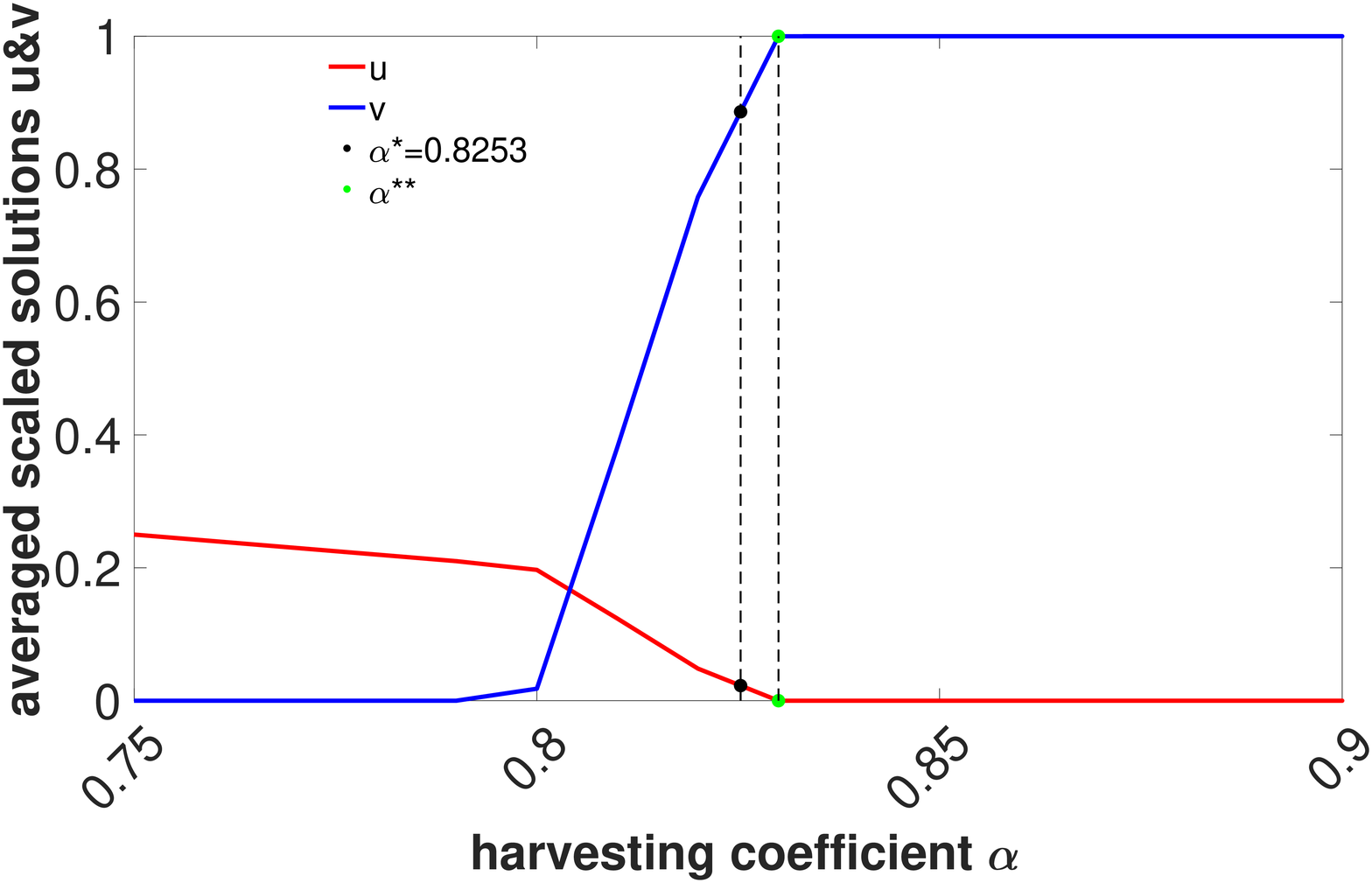}
\caption{Average solutions to \eqref{eq:main_problem} with $a\equiv b \equiv 1$, $Q\equiv1$ and $P = K = 2+\cos(\pi x),~r(x) \equiv 1.1$, $x\in(0,4)$ at  $t=T=2000$ are presented  for  $\beta=0,0.4,0.6,0.8$ (from left to right, top to bottom) with corresponding $\alpha^* = 0.0995, 0.4671, 0.6475, 0.8253$, respectively.}
\label{example_1_fig_2}
\end{figure}
\autoref{example_1_fig_3} represents dependence of the average population density on 
the harvesting rates $\alpha$ and $\beta$. 
\begin{figure}[ht]
	\centering
                \includegraphics[width=0.4\linewidth]{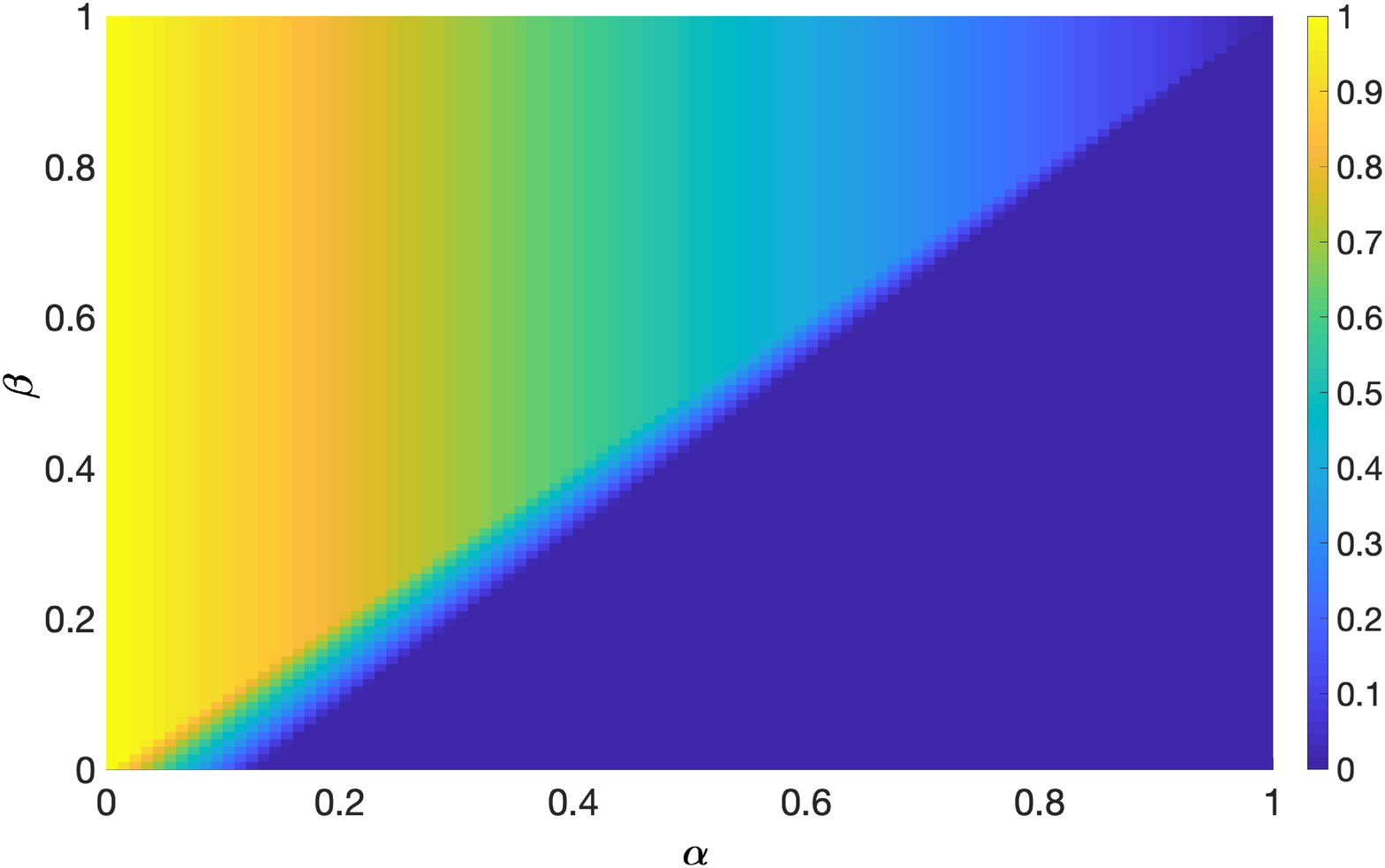}~~
                \includegraphics[width=0.4\linewidth]{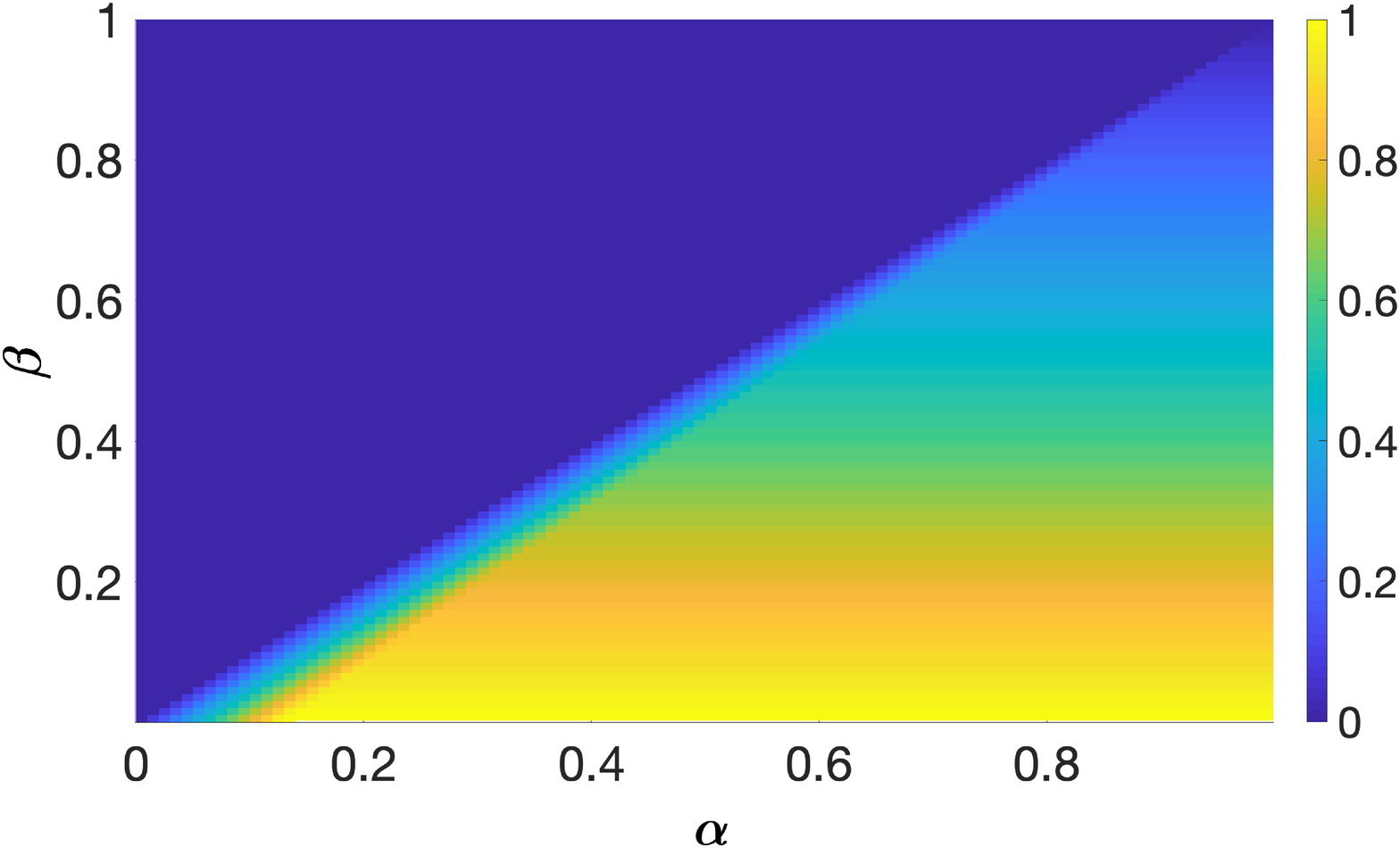}
	\caption{Average population densities 
       of $u$ (left) and $v$ (right) obtained from \eqref{eq:main_problem} for various harvesting 
rates $\alpha\in [0,1]$ and $\beta\in[0,1]$.}
	\label{example_1_fig_3}
\end{figure}
\end{Ex}

\begin{Ex}
\label{ex2}
Let us consider competition model \eqref{eq:main_problem}, 
where the carrying capacity is a combination of Gaussian functions and a positive constant $K(x) = 10e^{-12.5 
\pi^2(x-2)^2}-e^{-50 \pi^2(x-2)^2}+1$ with $P(x) \equiv K(x)$, $a(x)=b(x)=Q(x) = 1$ for all $x\in(0,4)$, $t>0$. Let 
$r(x)\equiv 1.1$,  
$\alpha,\beta \in (0,1)$, and set the initial conditions as $u_0=2.1=v_0$. 
Similarly to the previous example, without harvesting, there is a competitive exclusion of a regularly diffusing population, while harvesting of $u$ at small levels leads to coexistence, as shown in \autoref{example_2_fig_1}. 	
\begin{figure}[ht]
		\centering
		\includegraphics[width=0.4\linewidth]{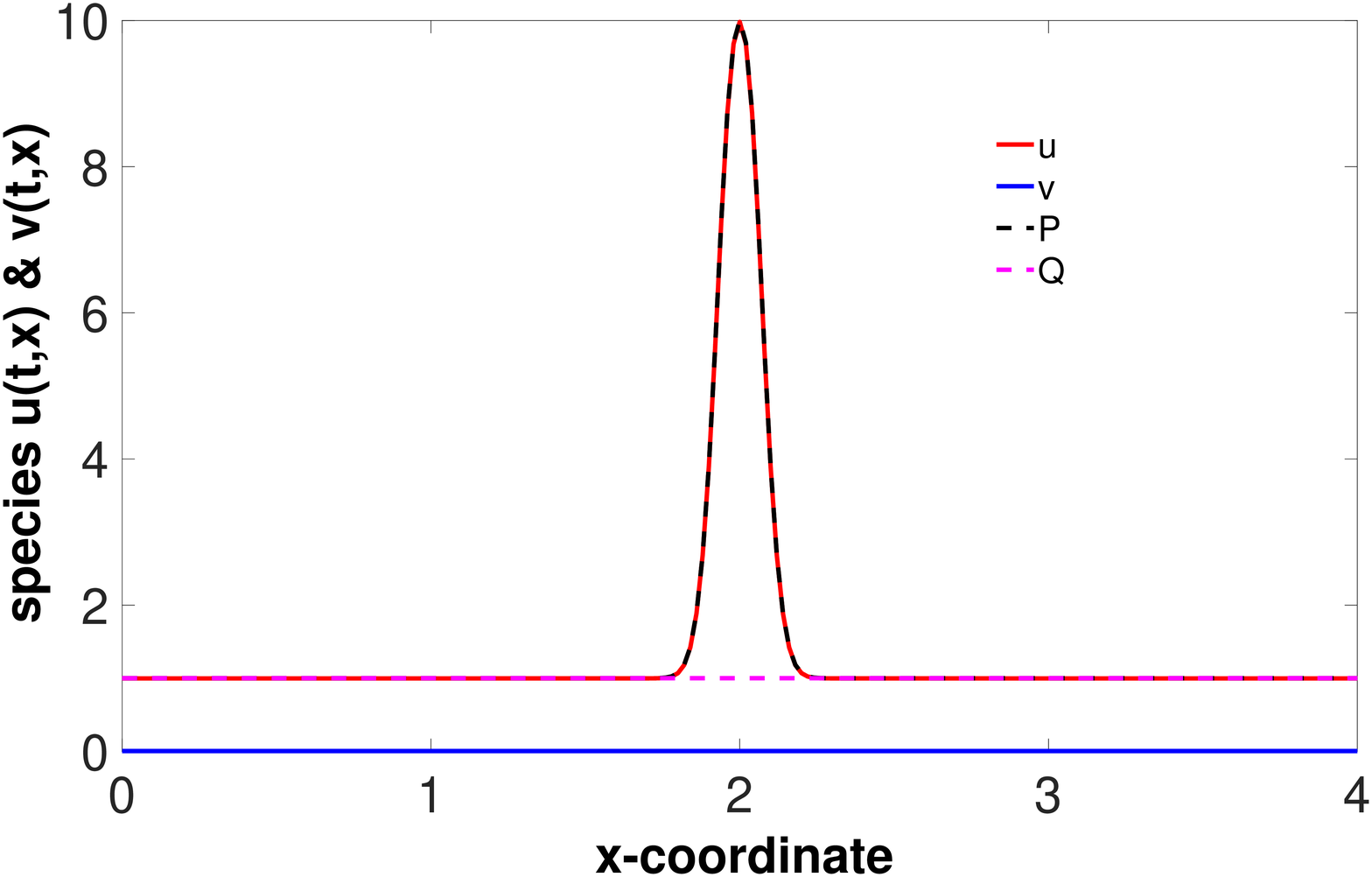}~~
			\includegraphics[width=0.4\linewidth]{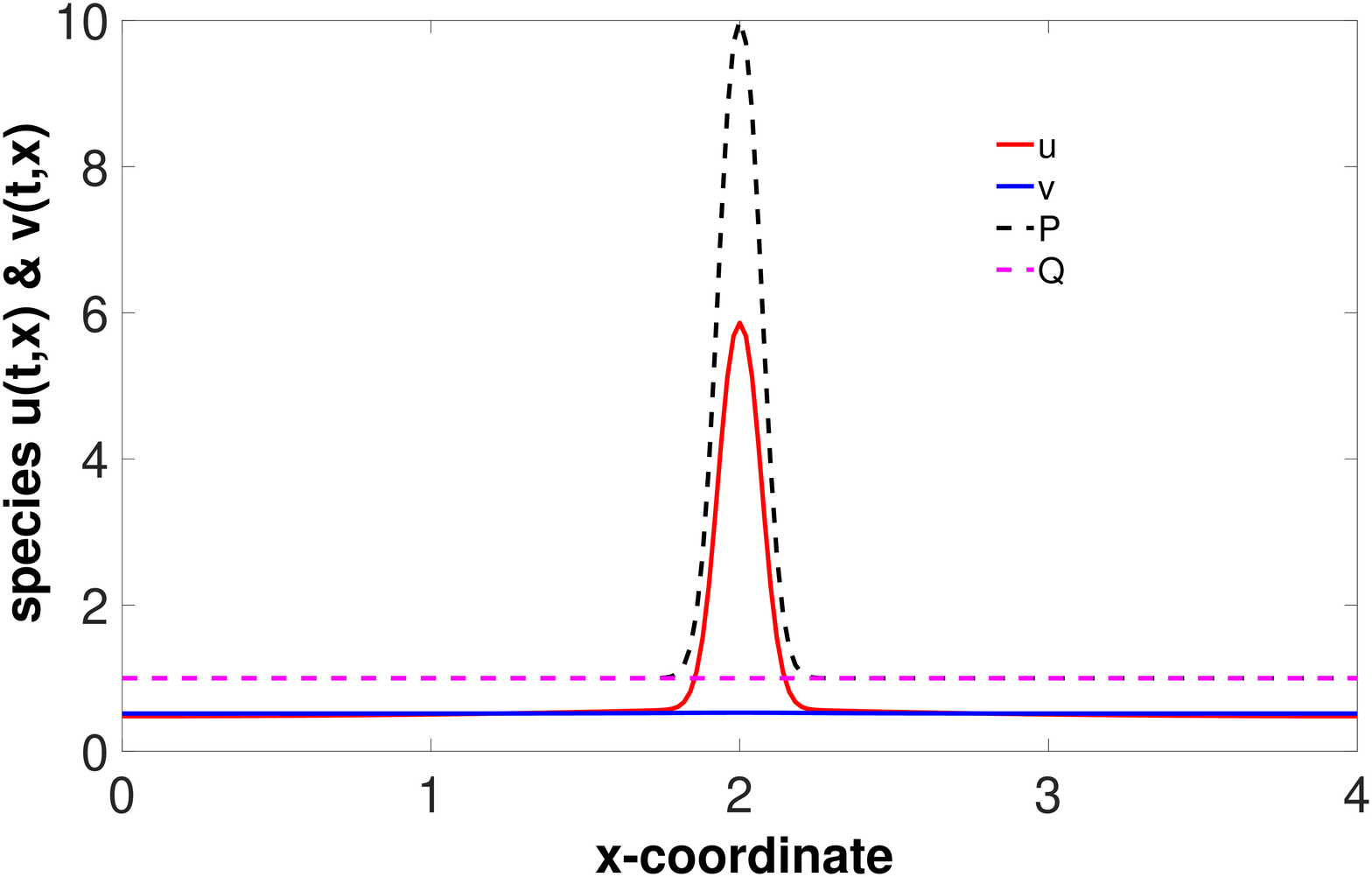}
		\caption{Spatial distributions of solutions to \eqref{eq:main_problem} at $t=T=2000$, with $P = K = 10e^{-12.5 \pi^2(x-2)^2}-e^{-50 \pi^2(x-2)^2}+1,a(x)=b(x)=Q(x) \equiv 1, u_0=v_0=2.1$, $x\in(0,4)$ and  (left) $\alpha=\beta=0$ , (right) $\beta=0$, $\alpha = 0.1$.}
		\label{example_2_fig_1}
	\end{figure}
We provide graphs of average stationary solutions as functions of $\alpha$ for fixed values 
$\beta$ (\autoref{example_2_fig_2}), using \eqref{alpha_star} to compute the upper bound $\alpha^*$.
\begin{figure}[!hb]
	\centering
		\includegraphics[width=0.4\linewidth]{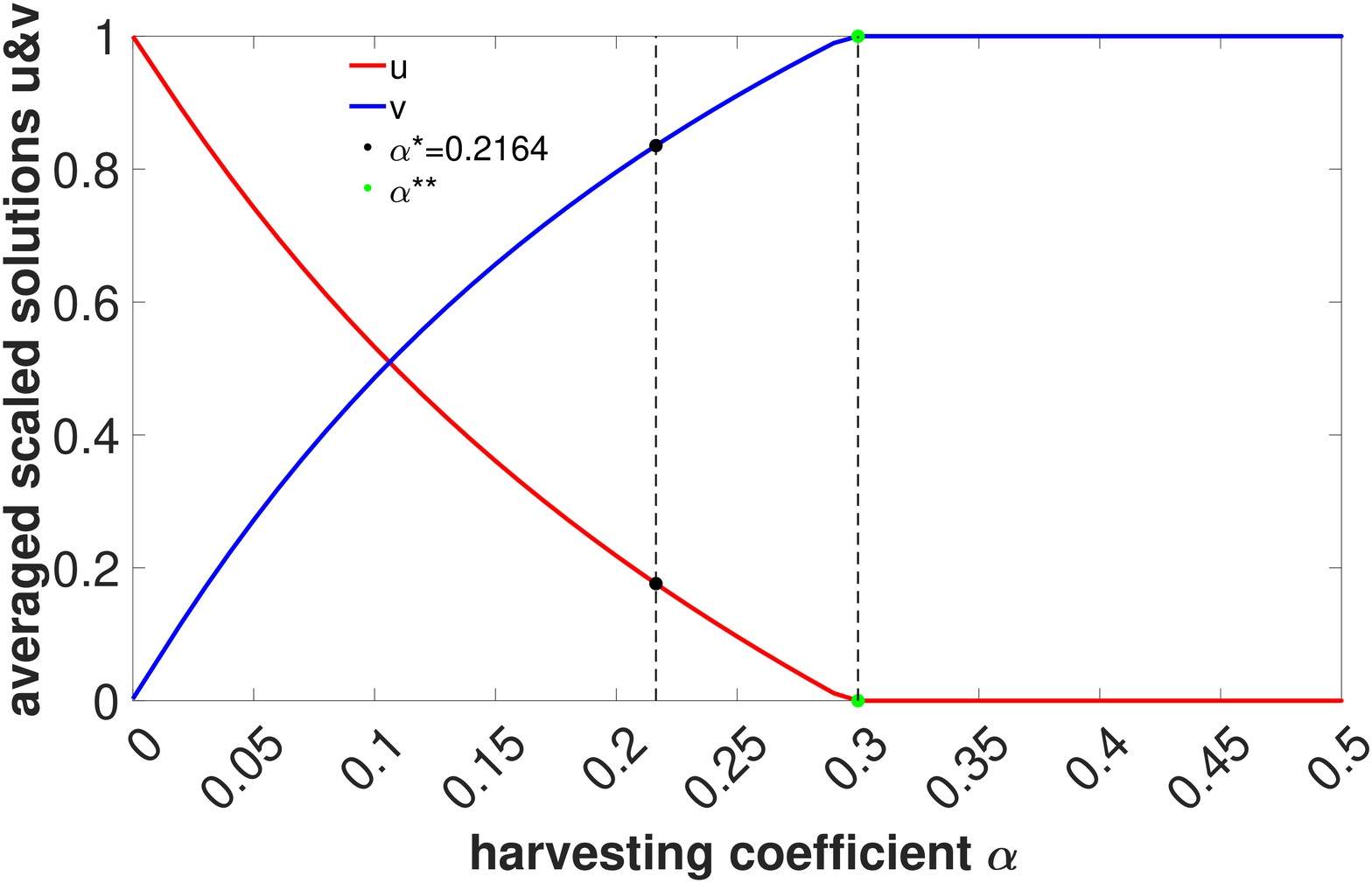} ~~
		\includegraphics[width=0.4\linewidth]{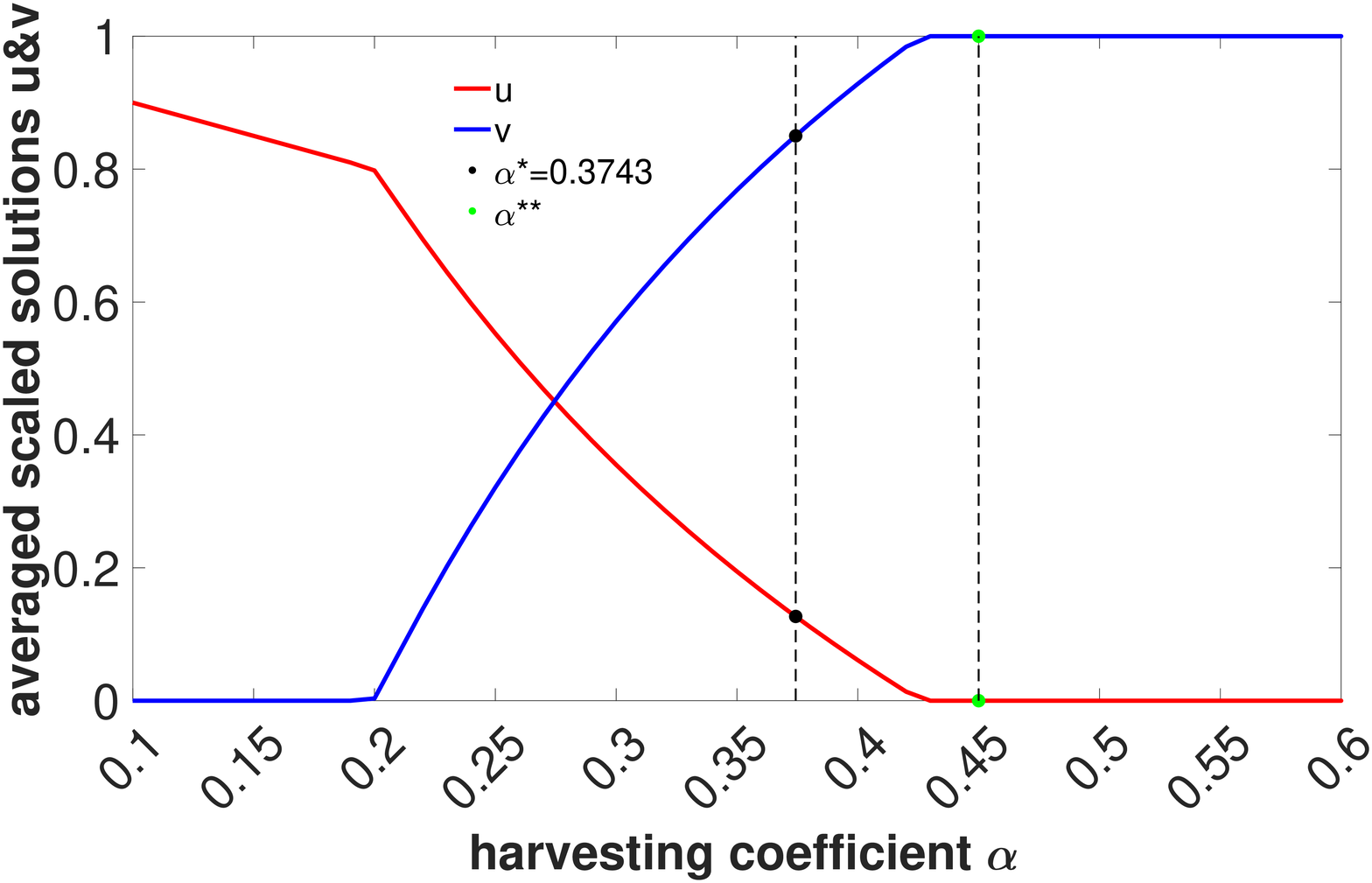}
		\includegraphics[width=0.4\linewidth]{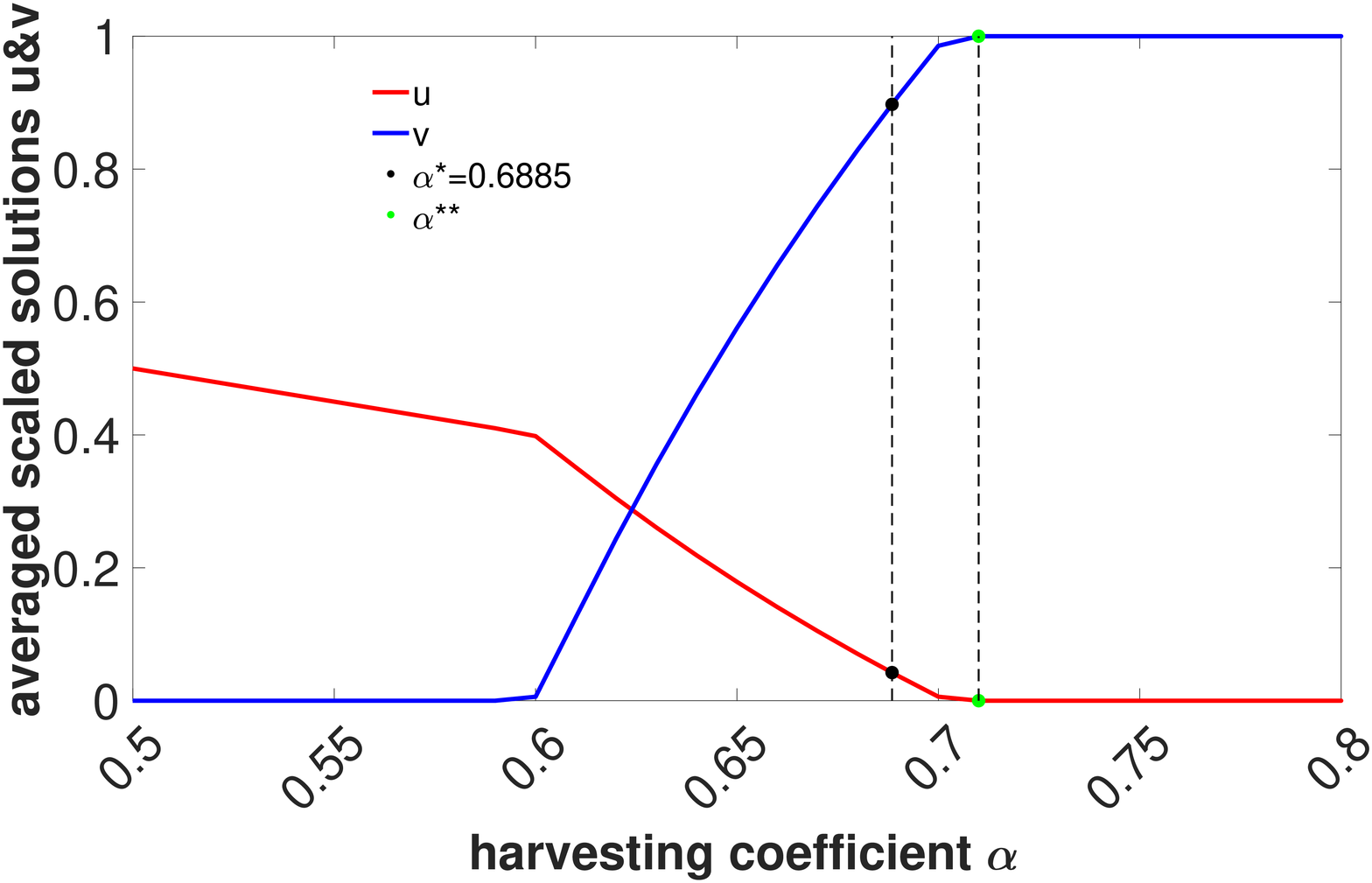} ~~
		\includegraphics[width=0.4\linewidth]{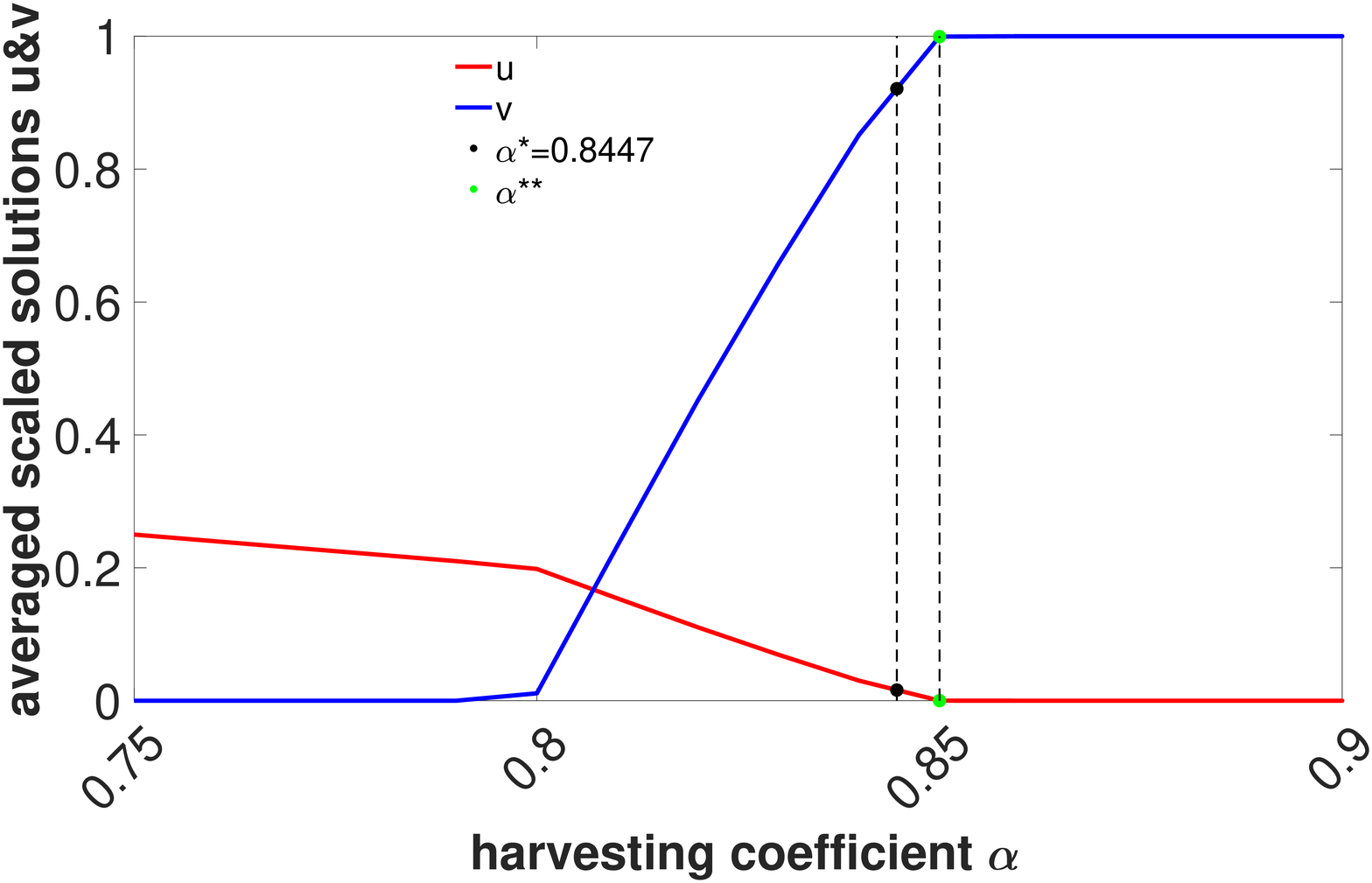}
	\caption{Average solutions to \eqref{eq:main_problem} with $a\equiv b \equiv 1$, 
$Q\equiv1$ and $P = K = 10e^{-12.5 \pi^2(x-2)^2}-e^{-50 \pi^2(x-2)^2}+1$, $r(x) \equiv 1.1$ for 
$x\in(0,4)$ at  $t=T=2000$ are presented  for  $\beta=0,0.2,0.6,0.8$ with respective 
$\alpha^* = 0.2164, 0.3743, 0.6885, 0.8447$.}
	\label{example_2_fig_2}
\end{figure}
\autoref{example_2_fig_3} presents dependency of the average population density on the harvesting rates $\alpha$ and $\beta$. 
\begin{figure}[ht]
	\centering
        \includegraphics[width=0.32\linewidth]{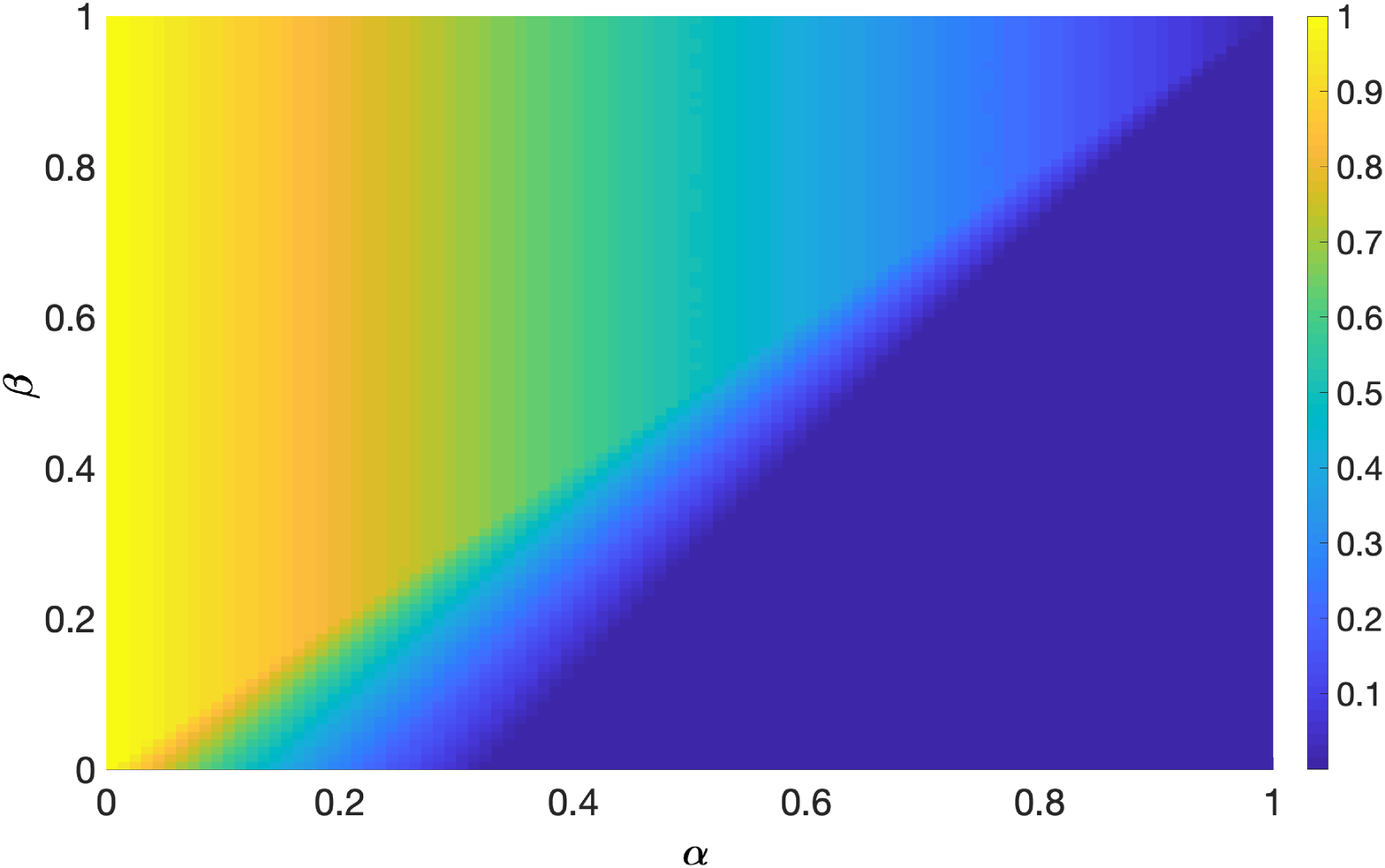}~~
        \includegraphics[width=0.32\linewidth]{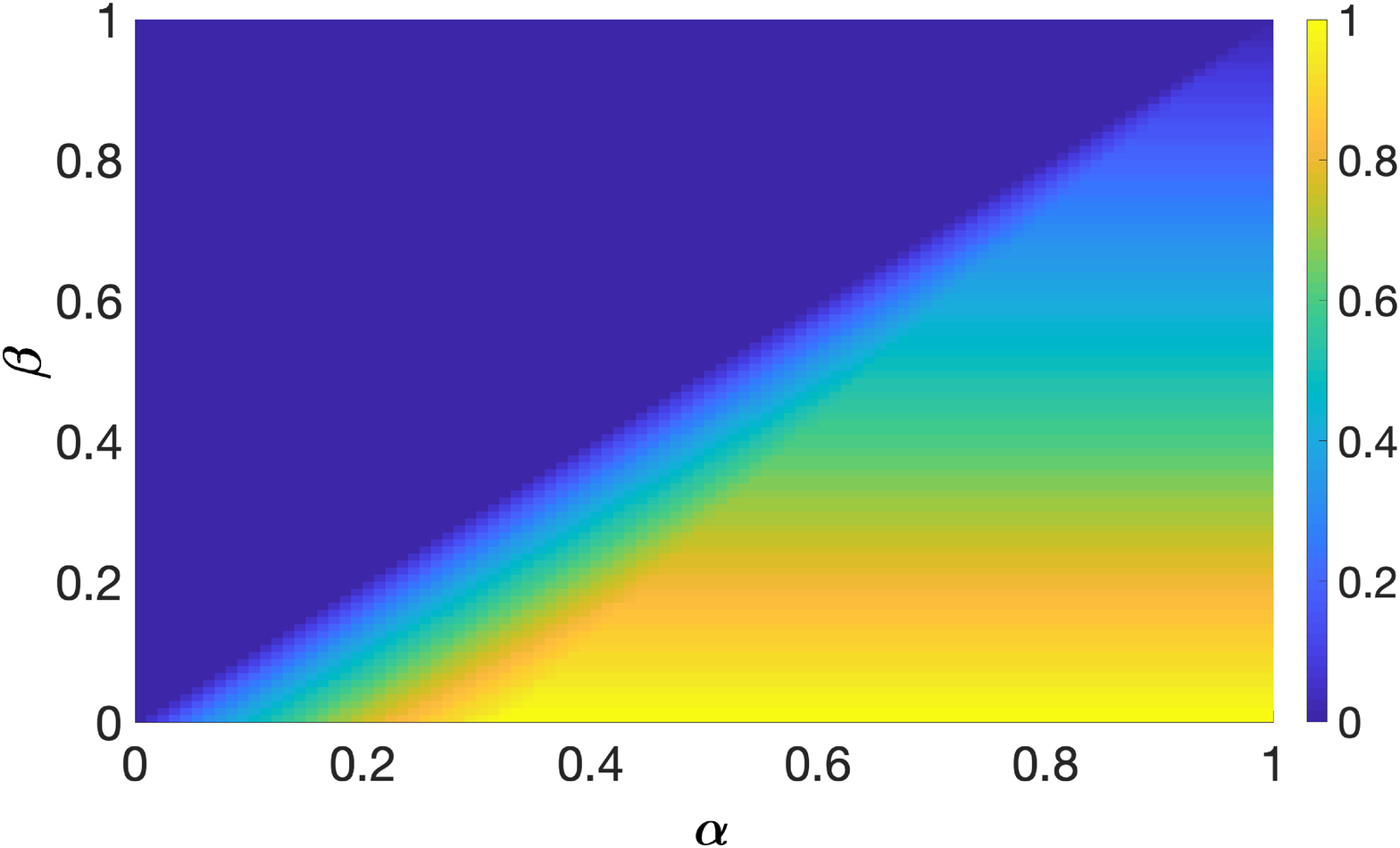}
		    \includegraphics[width=0.32\linewidth]{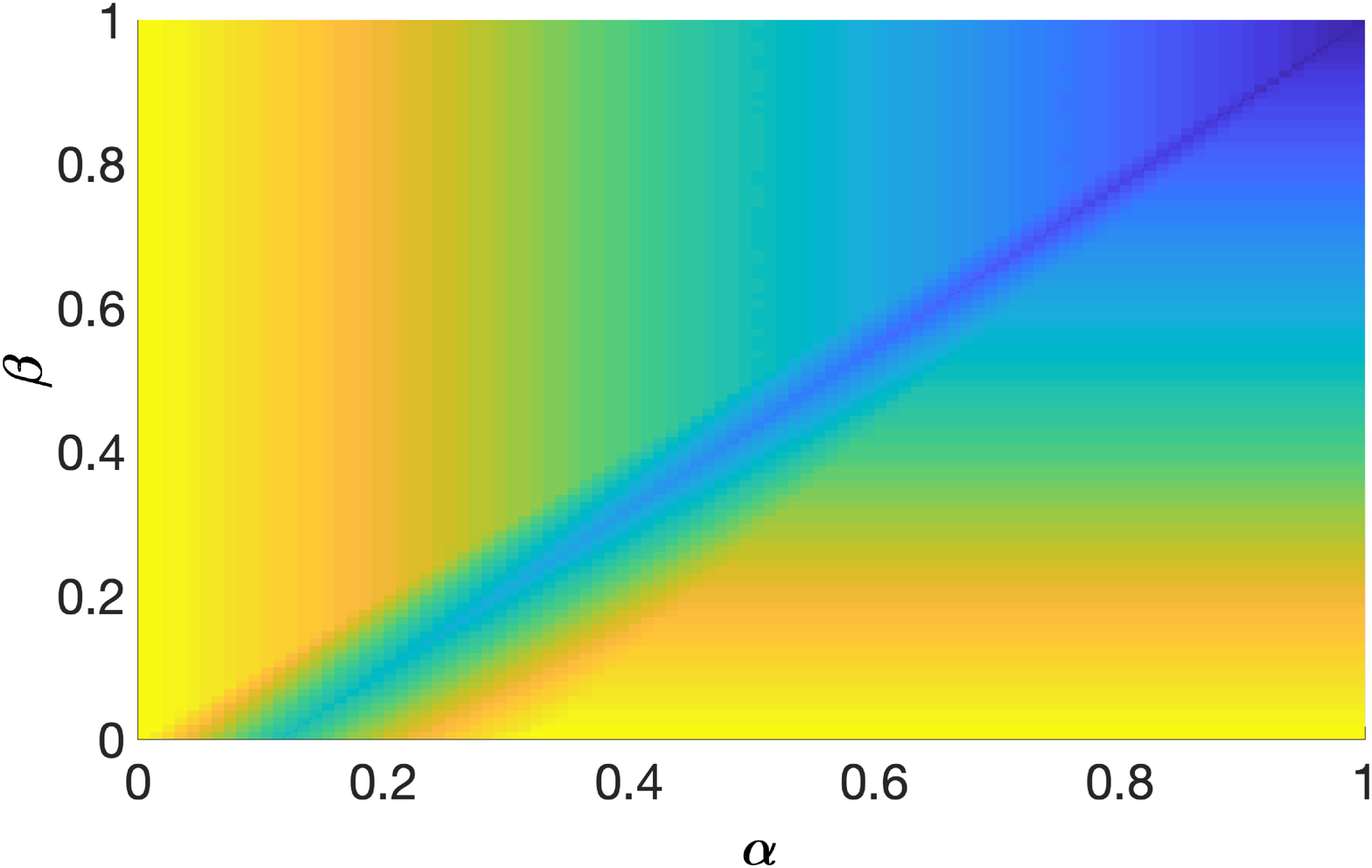}
	\caption{Average population densities of $u$ (left) and $v$ (middle) obtained from average solutions to 
\eqref{eq:main_problem} at $t=T=2000$ for various harvesting rates $\alpha\in [0,1]$ and $\beta\in[0,1]$,
	as well as both $u$ and $v$, where the darker line near bisect corresponds to coexistence.}
	\label{example_2_fig_3}
\end{figure}
\end{Ex}

\begin{Ex}
\label{ex3}
Consider the case where $P$ and $Q$ form an ideal free pair, concretely, $P=1.1+0.5\cos(\pi x)$,  $Q(x) = 0.9+0.5\cos(\pi x)$, 
$K(x) = 2+\cos(\pi x) = P(x)+Q(x)$, $x\in (0,4)$. The initial conditions are $u_0=2.1=v_0$. 
We set $r(x)\equiv 1.1$ and consider $\alpha,\beta \in (0,1)$. 
Note that $P$ and $Q$ are linearly independent. 
According to  \cite[Theorem 1]{CAMWA2016} the coexistence solution will be $(P(x),Q(x))$ in the absence of harvesting. To 
illustrate the results of Theorem~\ref{theorem_harv2}, Part 2,  for each fixed $\beta$ we compute the estimate $\alpha^{*}$ as 
in \eqref{c_ast_est}.
We compare population densities at time $t=T=2000$ (see \autoref{example_3_fig_1}) with different (zero 
and non-zero) values of harvesting, observing the achieved coexistence state.
\begin{figure}[ht]
		\centering
		\includegraphics[width=0.4\linewidth]{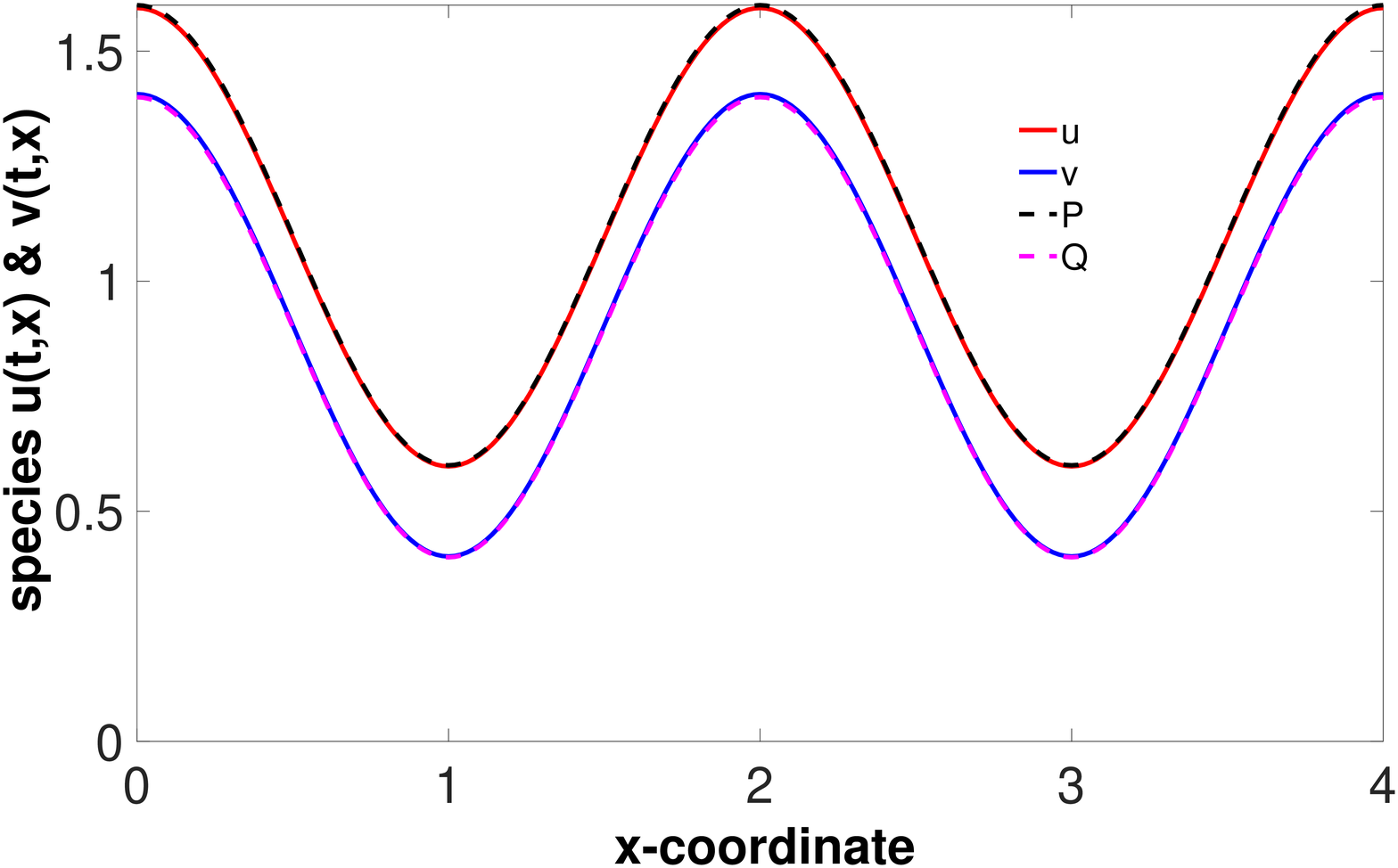}~~
		\includegraphics[width=0.4\linewidth]{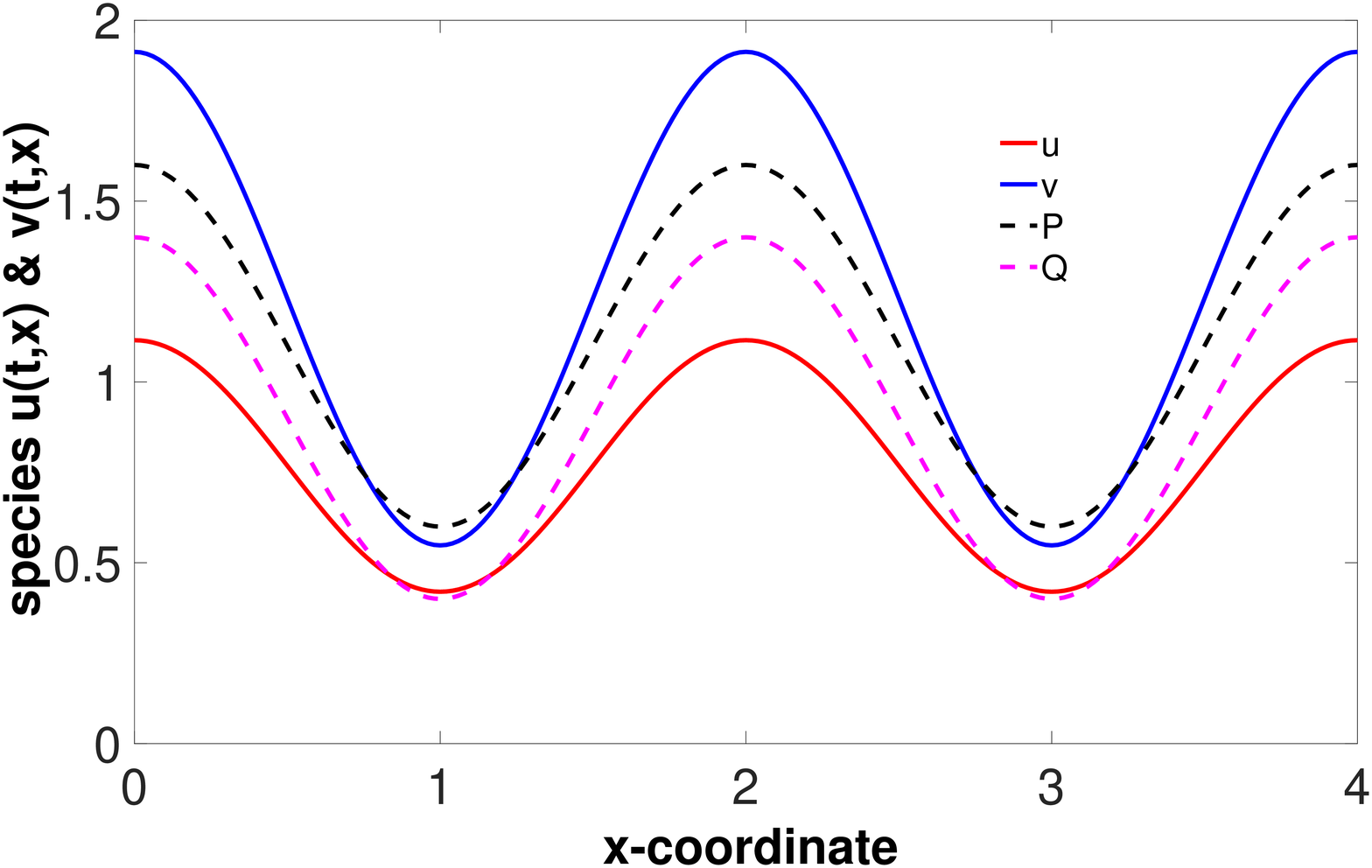}
		\caption{Representations of average solutions $u$ and $v$ of \eqref{eq:main_problem} at  $t=T=2000$ for $\alpha=\beta = 0$ (left) and $\alpha=0.001,~\beta = 0$ (right). We set $u_0=v_0=2.1,~P=1.1+0.5\cos(\pi x), ~ Q(x) = 0.9+0.5\cos(\pi x), ~K(x) = 2+\cos(\pi x) = P(x)+Q(x)$, $r(x)\equiv 1.1$, $x \in (0,4)$.}
		\label{example_3_fig_1}
\end{figure}
In \autoref{example_3_fig_2} the case of solutions at  $t=T=2000$ as functions of harvesting $\alpha,\beta$ is illustrated. 
We construct the dependence of these functions on the value of $\alpha$ for several fixed $\beta\in[0,1)$ indicating the value of $\alpha^{*}$ computed by \eqref{c_ast_est} and $\alpha^{**}$ as the maximal $\alpha_1$ when coexistence occurs computed numerically.
\begin{figure}[ht]
	\centering
	\includegraphics[width=0.4\linewidth]{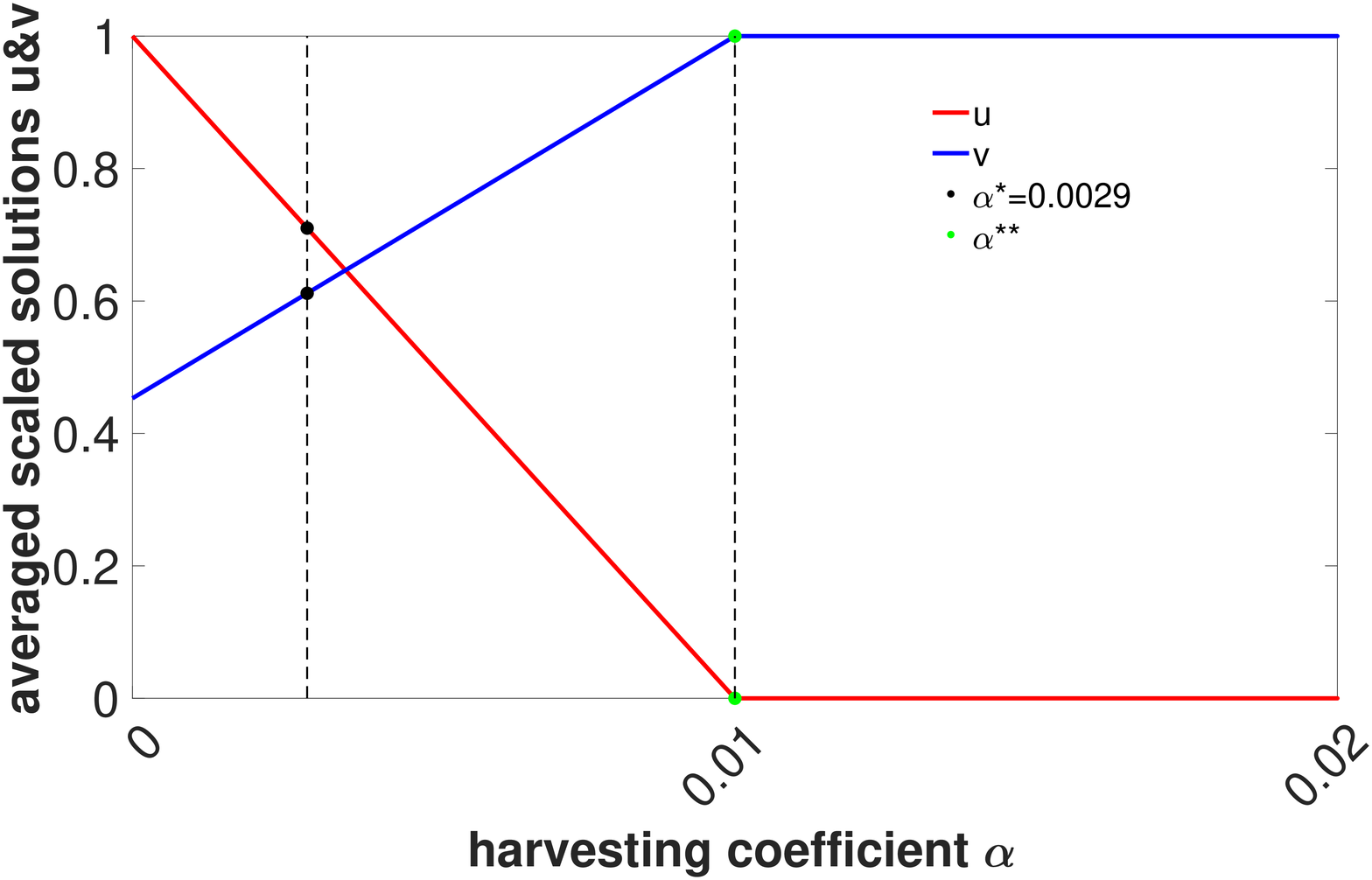} ~~
	\includegraphics[width=0.4\linewidth]{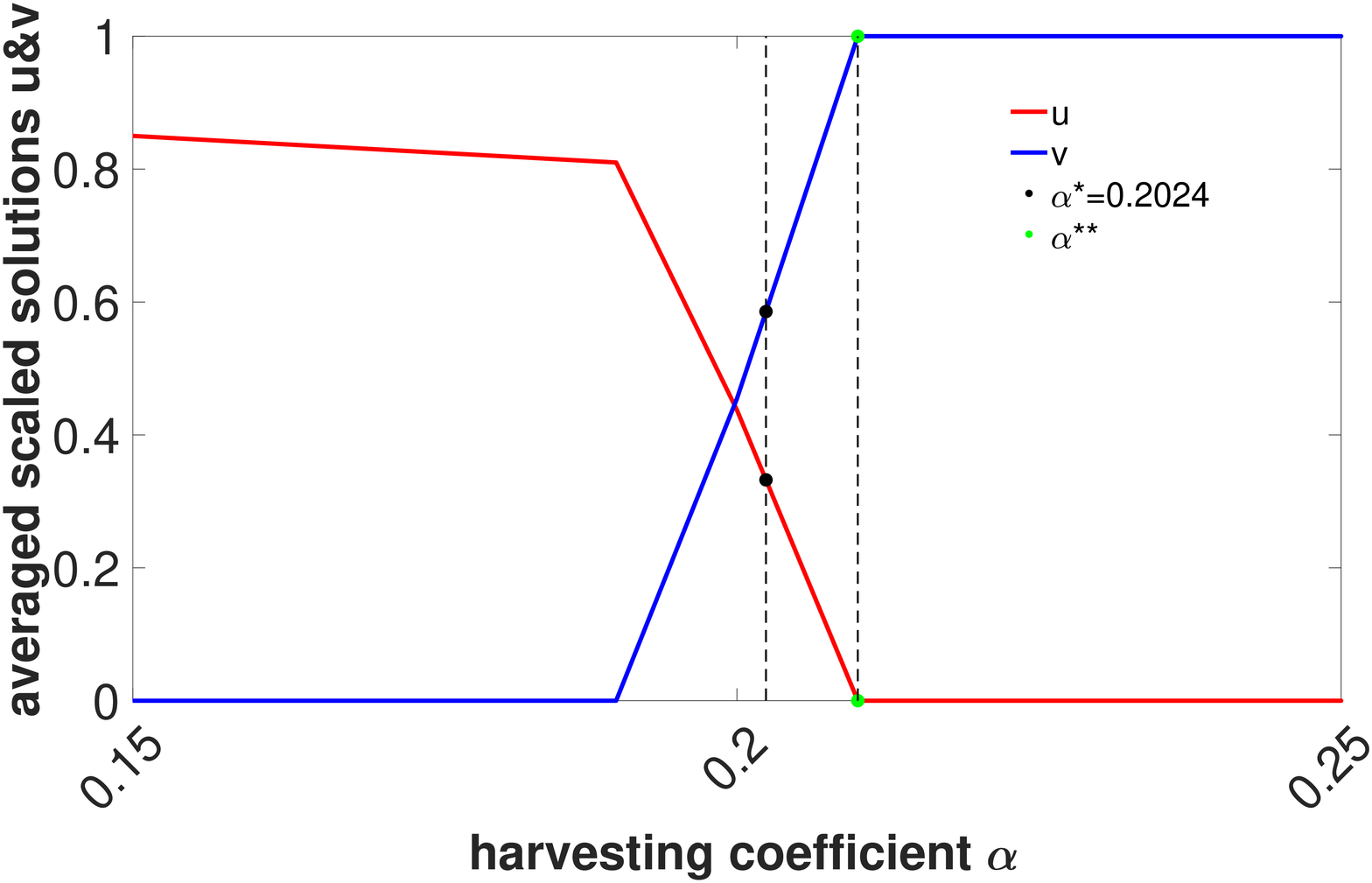}
	\includegraphics[width=0.4\linewidth]{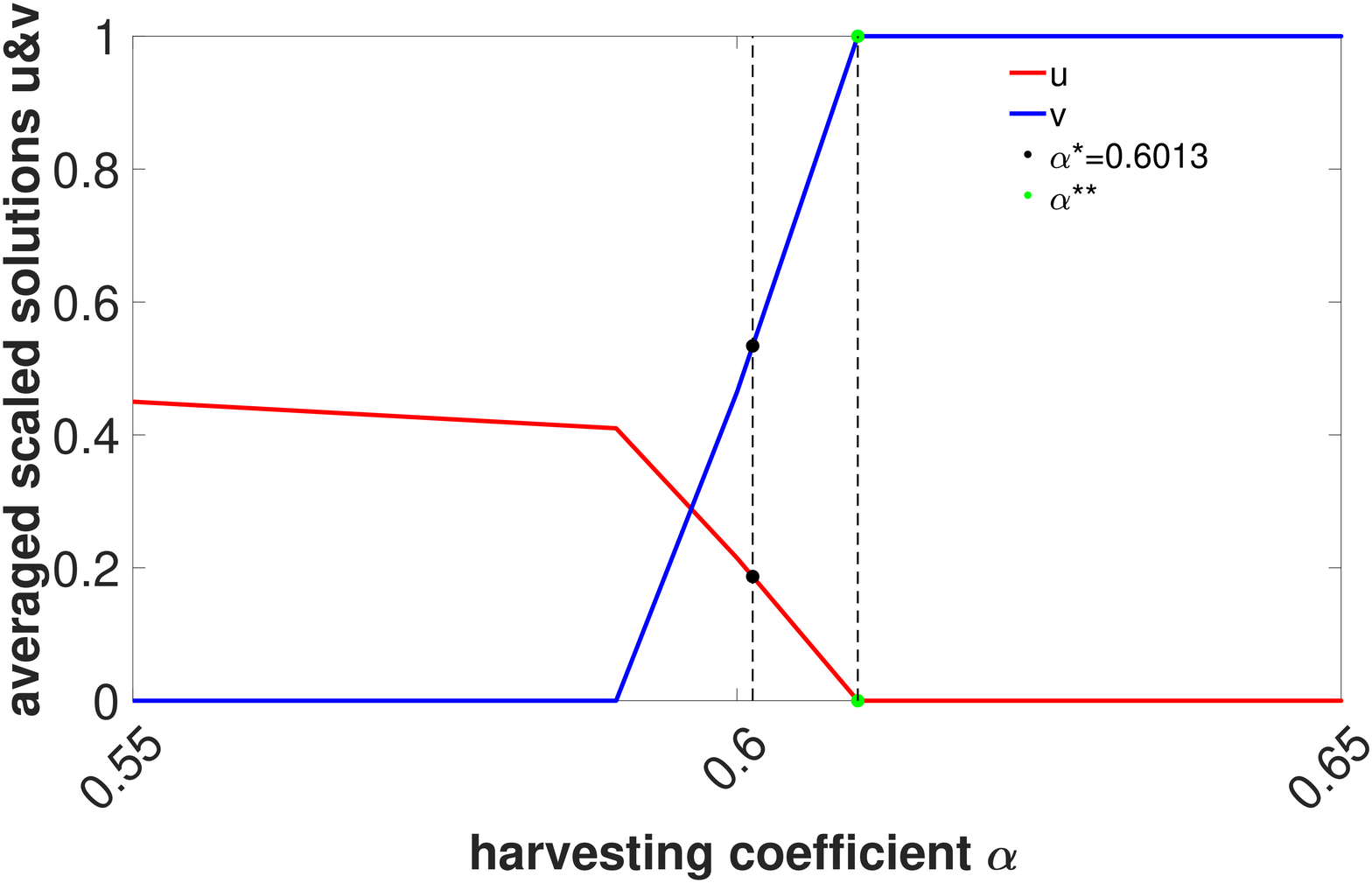} ~~
	\includegraphics[width=0.4\linewidth]{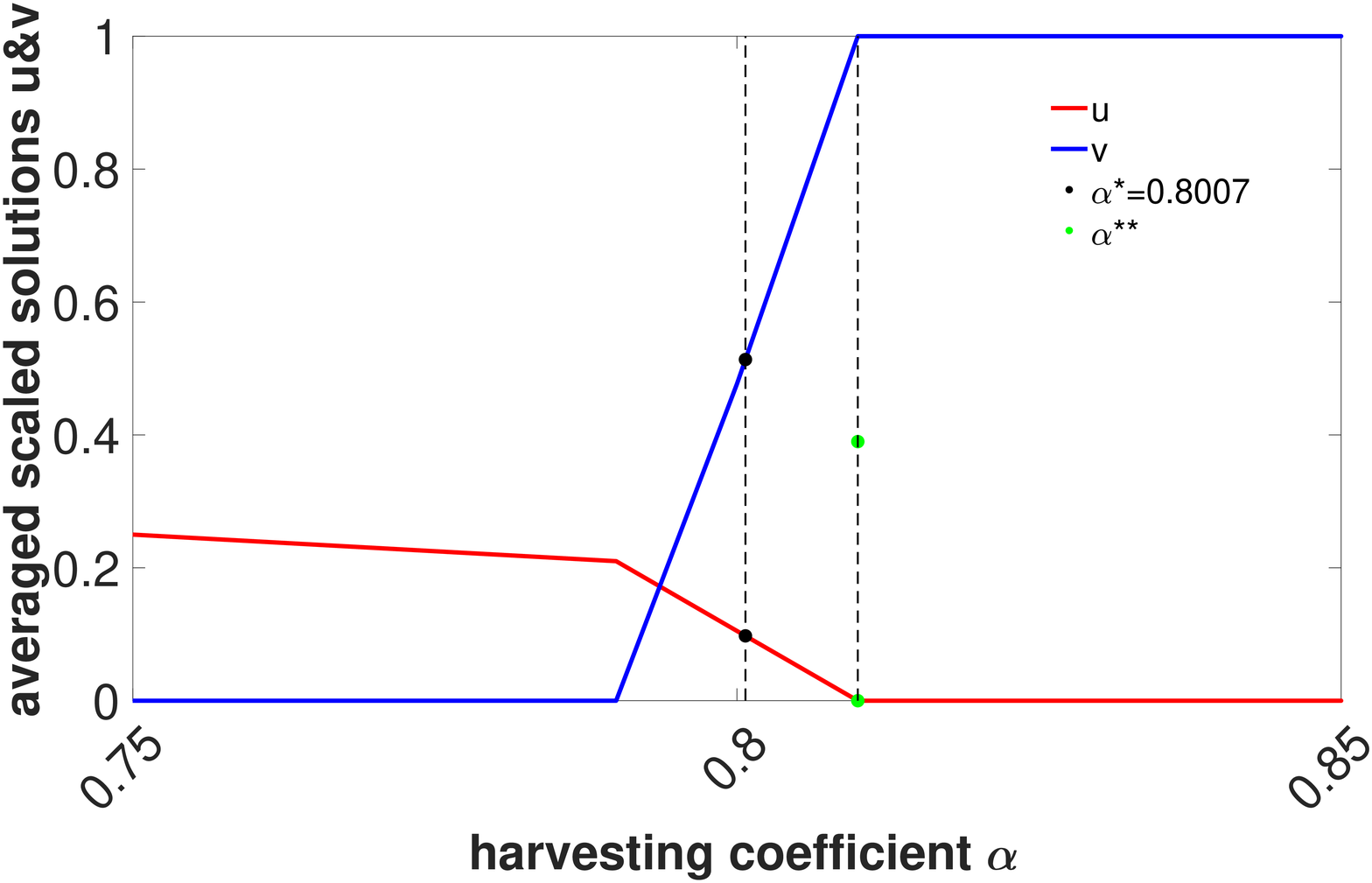}
	\caption{Average solutions of \eqref{eq:main_problem} with $P=1.1+0.5\cos(\pi x), ~ Q(x) = 0.9+0.5\cos(\pi x), ~K(x) = 2+\cos(\pi x) = P(x)+Q(x)$, $r(x)\equiv 1.1$ as functions of $\alpha$ for a fixed $\beta$ at large $t=T=2000$ are presented  for  $\beta=0,0.2,0.6,0.8$ with corresponding $\alpha^{*} = 0.0029, 0.2024, 0.6013, 0.8007$, respectively.}
	\label{example_3_fig_2}
\end{figure}
In \autoref{example_3_fig_3}, we compare average population densities $u$ and $v$ at time $t=T=2000$ as functions of both 
$\alpha$ and $\beta$.
\begin{figure}[ht]
	\centering
	\includegraphics[width=0.4\linewidth]{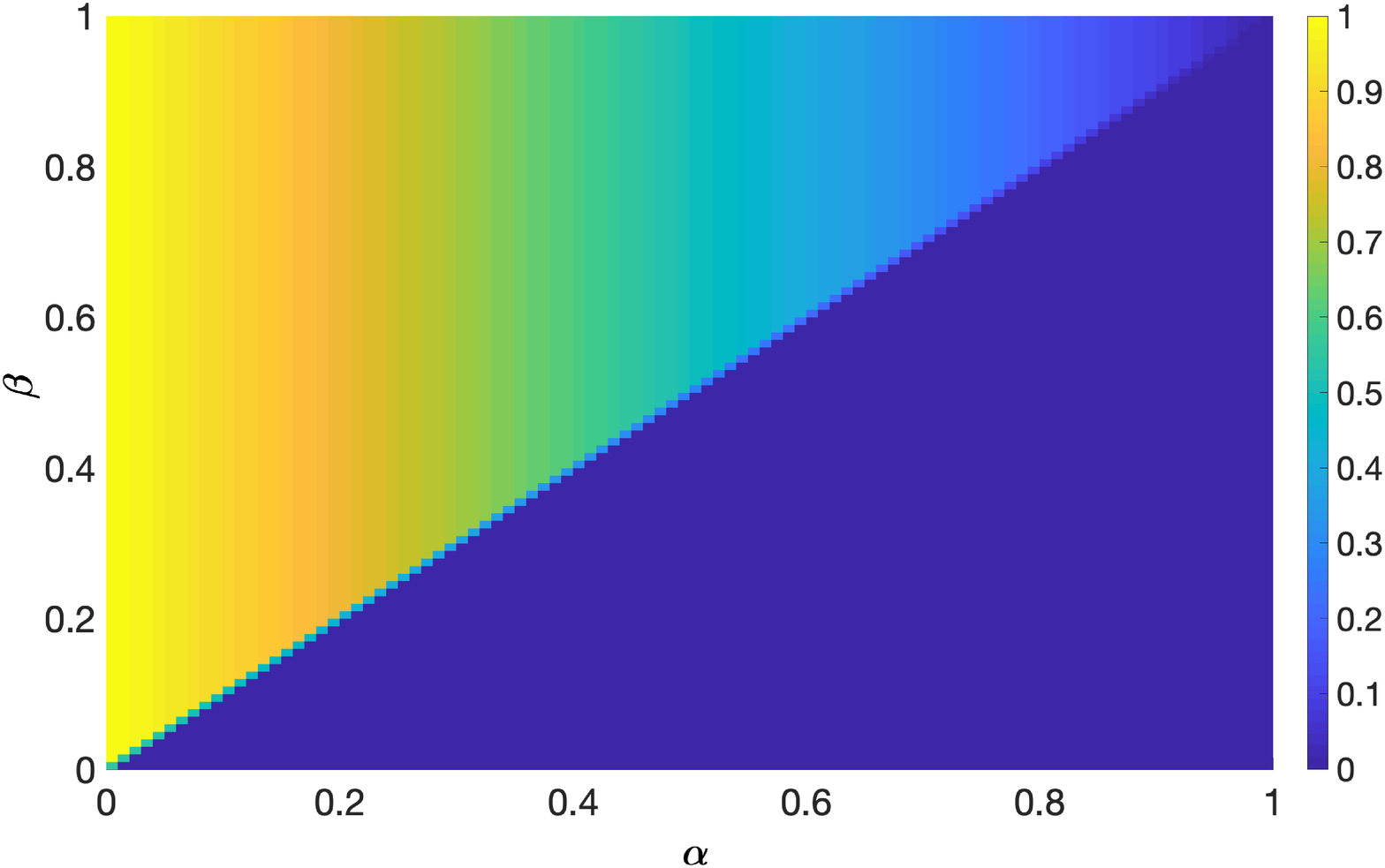}~~
	\includegraphics[width=0.4\linewidth]{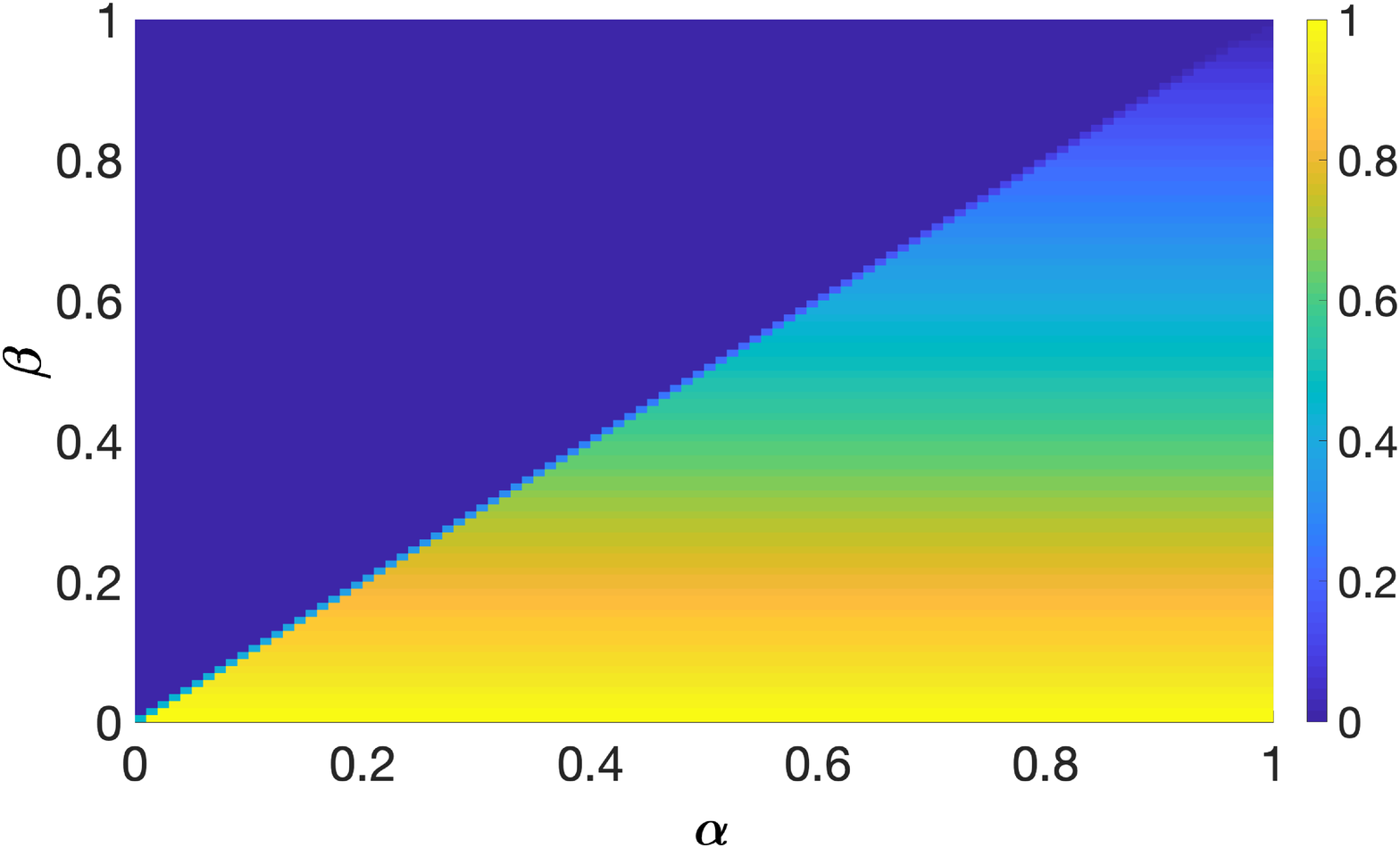}
	\caption{Average population densities of $u$ (left) and $v$ (right) as functions of harvesting rates $\alpha\in [0,1]$ and $\beta\in[0,1]$.}
	\label{example_3_fig_3}
\end{figure}
	
\end{Ex}

\begin{Ex}
\autoref{heatmap} illustrates how the values of $\alpha$ and $\beta$ that provide coexistence are distributed for an 
ideal free pair $P(x) = 1 + 10e^{-12.5\pi^2(x-2)^2}$, $Q(x) = e^{-50\pi^2(x-2)^2}+2$, $K = 
10e^{-12.5\pi^2(x-2)^2}+e^{-50\pi^2(x-2)^2}+3= P+Q$, 
$x\in(0,4)$. Similarly, in \autoref{heatmap_2} we show distribution of coexistence promoting rates $\alpha$ and $\beta$ 
for the case when $P(x) = K(x) = 10e^{-12.5\pi^2(x-2)^2}+e^{-50\pi^2(x-2)^2}+3$, $Q(x) = e^{-50\pi^2(x-2)^2}+2$, $x\in (0,4)$. 
Note that in \autoref{heatmap} there is a region of values near $(0,0)$ for which coexistence is preserved, 
while in the case depicted on \autoref{heatmap_2}, only increasing $\alpha$ allows to avoid competitive exclusion ($\beta=0$). 	
	\begin{figure}[ht]
		\centering
		\includegraphics[width=0.4\linewidth]{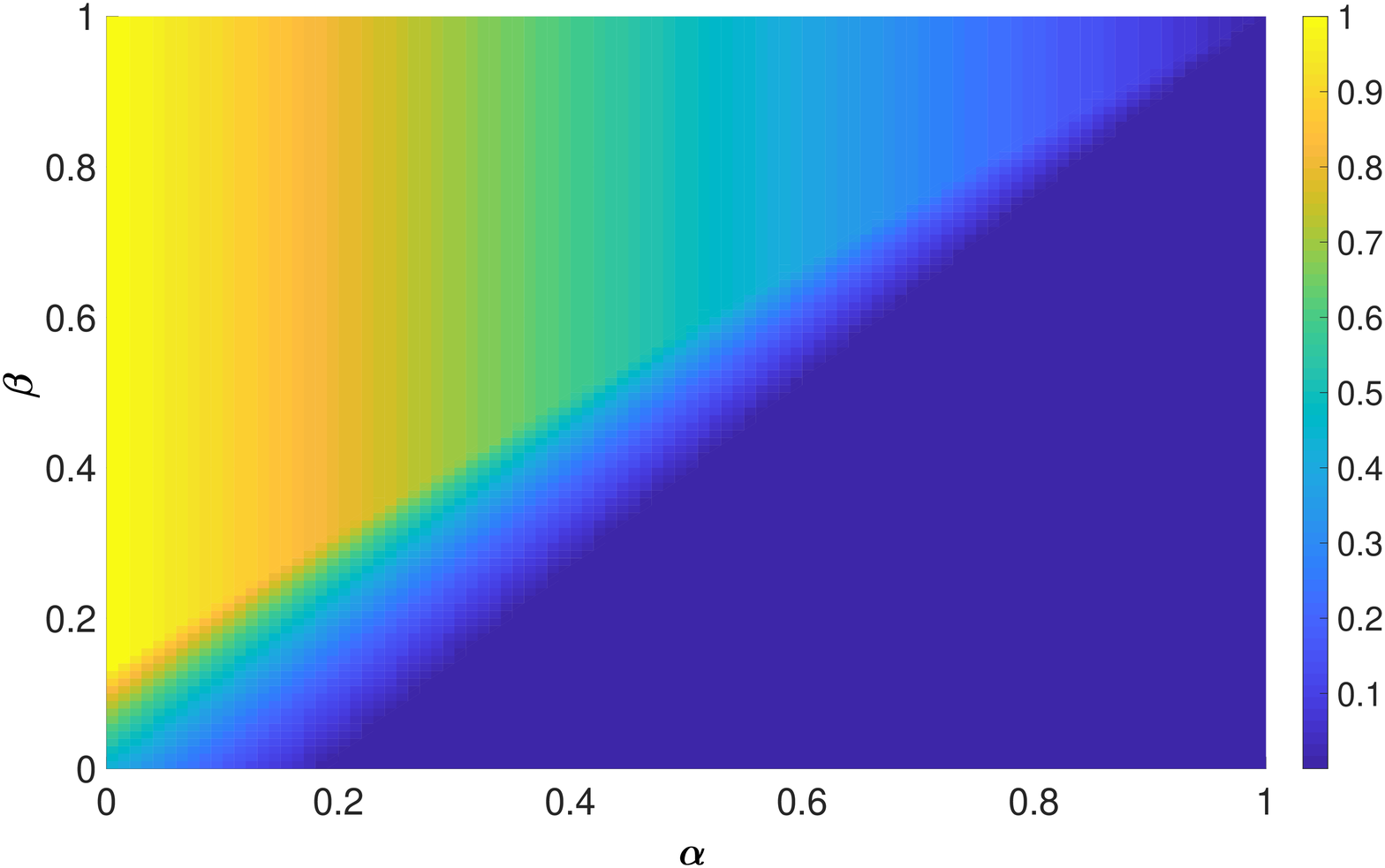}
		\includegraphics[width=0.4\linewidth]{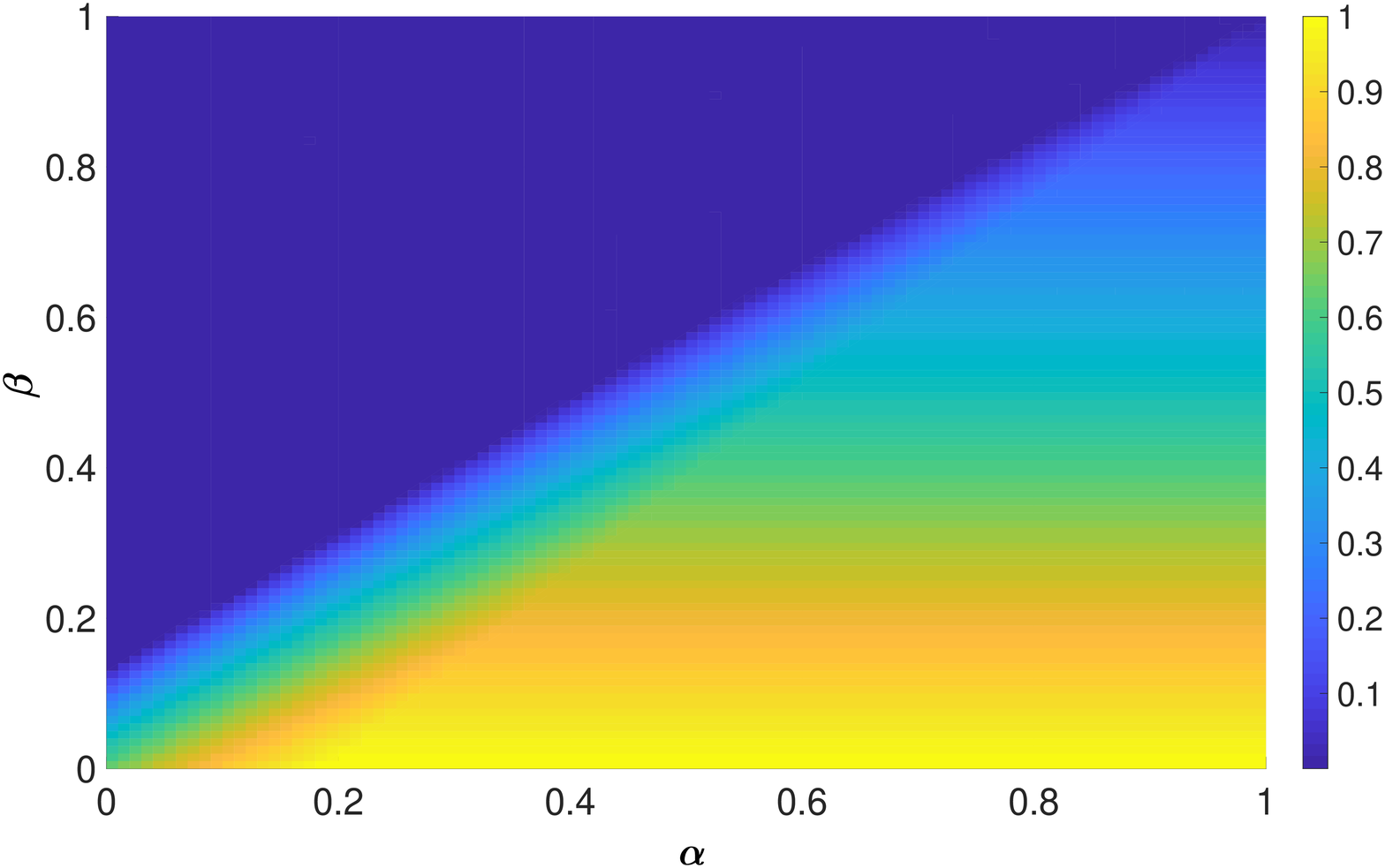}
		\caption{$P(x) = 1 + 10e^{-12.5\pi^2(x-2)^2}$, $Q(x) = e^{-50\pi^2(x-2)^2}+2$, $K = 10e^{-12.5\pi^2(x-2)^2}+e^{-50\pi^2(x-2)^2}+3= P+Q$, $x\in (0,4)$. The band surrounding the bisect of the first quadrant corresponds to harvesting rates that support coexistence.}
		\label{heatmap}
	\end{figure}
\begin{figure}[ht]
	\centering
	\includegraphics[width=0.4\linewidth]{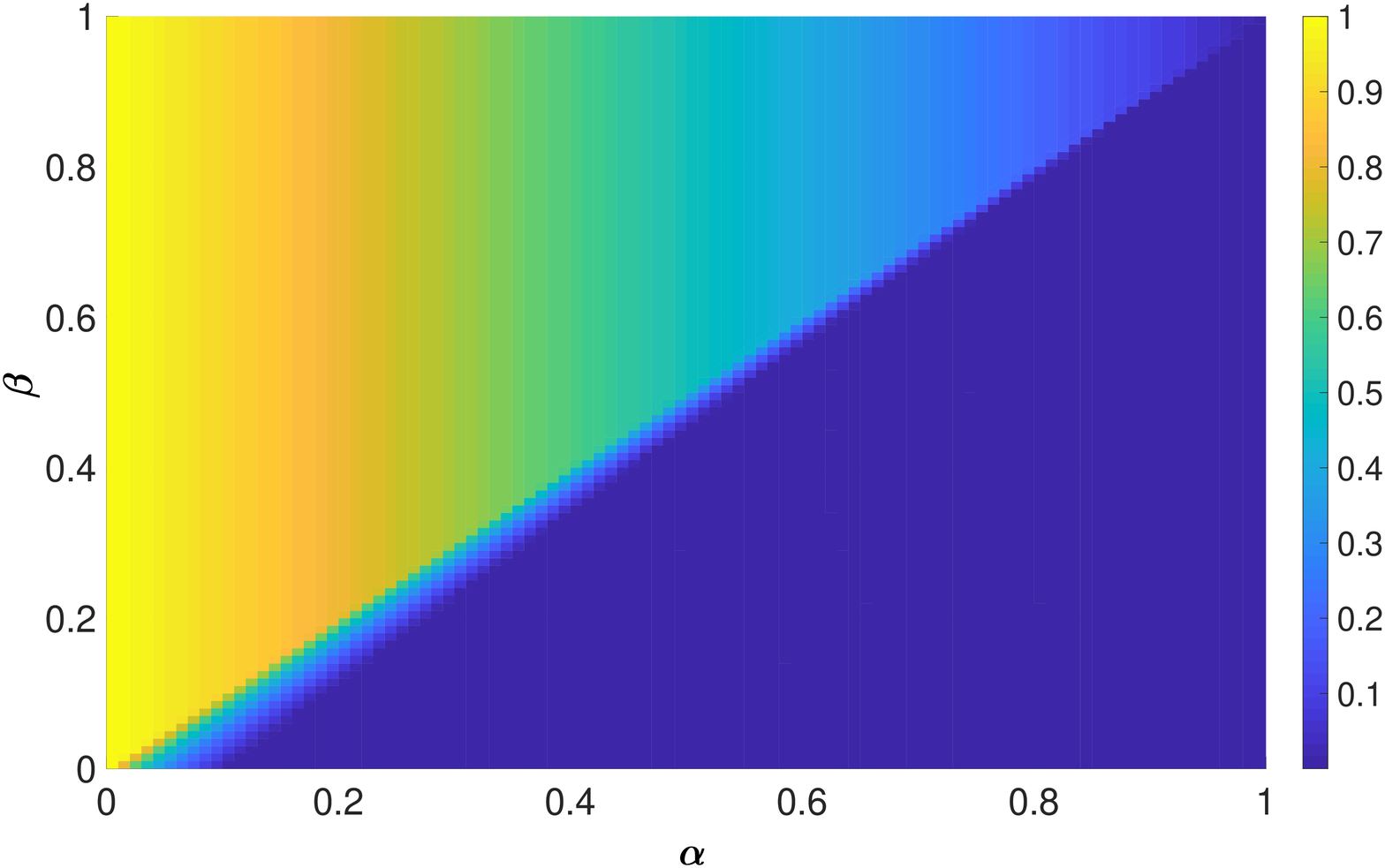}
	\includegraphics[width=0.4\linewidth]{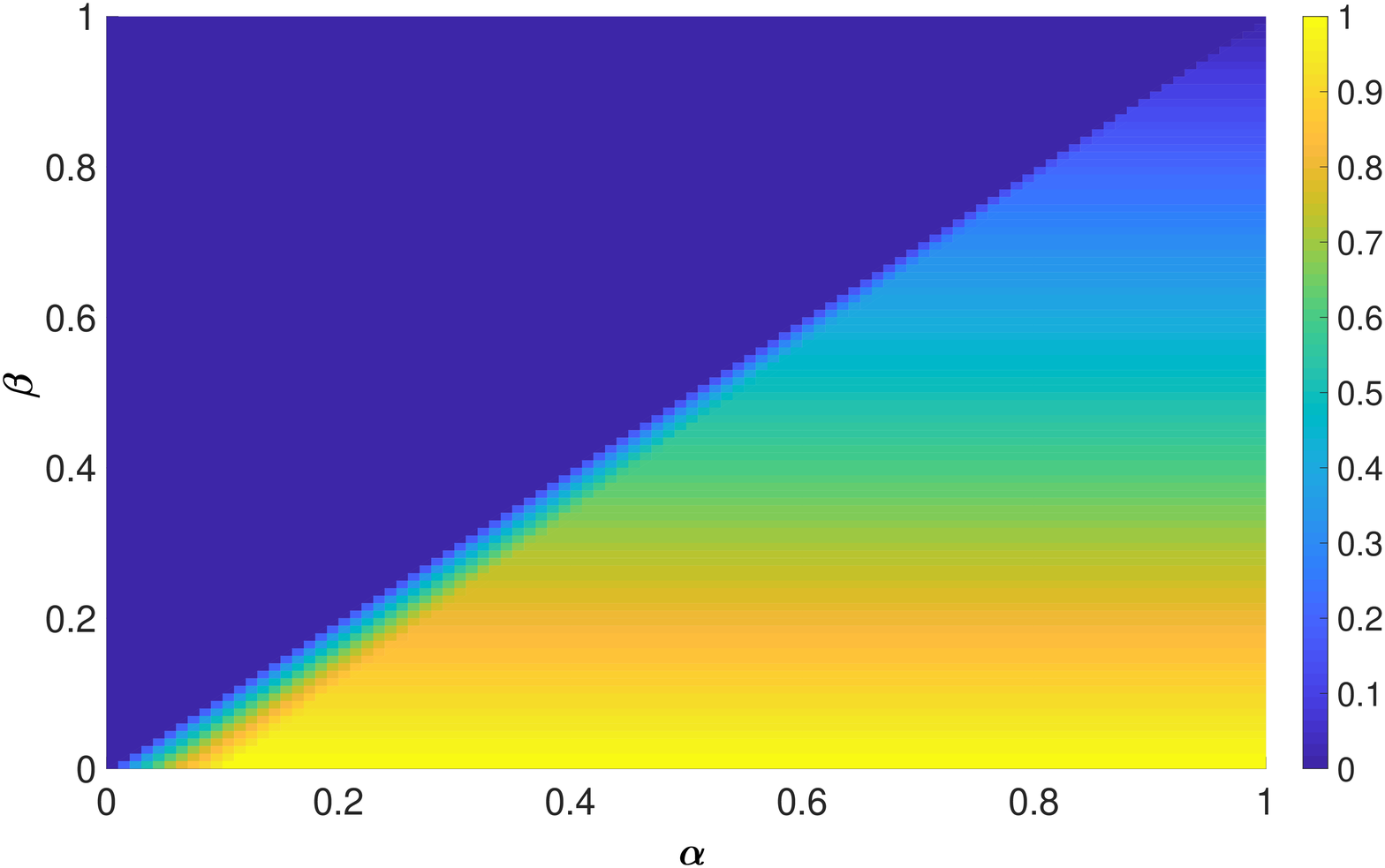}
	\caption{$P(x) = K(x) = 10e^{-12.5\pi^2(x-2)^2}+e^{-50\pi^2(x-2)^2}+3$, $Q(x) = e^{-50\pi^2(x-2)^2}+2$, $x\in (0,4)$, 
with the band around the bisect corresponding to 
coexistence.}
	\label{heatmap_2}
\end{figure}
\end{Ex}

\section{Discussion}
\label{sec:discuss}

The results obtained can shed some light on the species invasion management.
If a more efficient diffuser is an invader, all other parameters being the same, without harvesting or for the same or higher level of harvesting of the resident population, its habitat will be completely populated by invaders bringing  
the original species to extinction. However, any prevalence in harvesting of the invader can guarantee survival for the resident population; moreover, some estimates were presented when the invader goes extinct.
If a resident is a more efficient diffuser, its habitat can be invaded only with more intensive harvesting of the resident 
population. An estimate on harvesting rate is given when both species should coexist, in spite of the fact that the resident is 
exploited more than the invader. An additional estimate provides the level of harvesting guaranteeing that the habitat is completely invaded bringing the resident to extinction. 
If invader-resident form an ideal free pair, for close levels of harvesting coexistence is guaranteed, 
while harvesting levels were evaluated bringing the invader (or the resident) to extinction.

The estimation for the harvesting effort leading to coexistence is important, as it may give an indication 
which harvesting levels lead to sustainability of the exploited population. 
However, the proposed estimate, though approaching the switch value as the harvesting effort for the competitor 
increases, is still far from being sharp. 
Also, if culling is applied to bring one of the species to extinction, the minimal required reduction rate is overestimated.
  
The paper also explores the situation when, in the absence of harvesting, the two species coexist, based on partial specialization (expressed in the dispersal strategies) and  resources sharing. In this case, once the exploitation rate of one of the populations is fixed, we obtain a limitation on the harvesting level of the other population, such that the two species still coexist. Usually, this assumes a close level of harvesting for both species.
If there is an ideal free pair, harvesting of one of the species has to be high enough for
a possibility of the harvested species to be brought to extinction. For substantial efforts, even some smaller harvesting rates for the other species can be allowed, still leaving this population as the only survivor in the competition.
Also, as a harvested system is reduced to a model with modified growth rates and carrying capacities, we explored the influence 
of these parameters, changed proportionally, on the competition outcome.

The general approach to harvesting problems allows to extend research of the present paper in the following direction.
Here we considered external influence in the form of harvesting only.
More powerful control would involve stocking along with harvesting. For example, if a more efficient diffuser tries to invade a habitat of a less efficient resident, stocking of the resident species can prevent it from extinction.  
If a resident is in additional disadvantage being a preferred prey for a local predator, an additional stocking effort is required to avoid its extinction.
If stocking is allowed in addition to harvesting, for the ideal free pair and any (even zero) harvesting effort,
the population can be brought to extinction by a competitor having an excess of external resources. 
These results, we believe, can be justified in the same way as Theorems~\ref{theorem_harv1} and \ref{theorem_harv2}. 
{\bf While numerical simulations illustrate global stability of a unique coexistence equilibrium, whenever both 
semi-trivial equilibria is unstable, we do not state its uniqueness and global attractivity. Establishing this 
fact is still an open problem.
}

As a combination of the present paper and \cite{BKK2015}, it is possible to incorporate harvesting in the system where
one of the species more efficiently uses resources than the other one expressed in a higher carrying capacity. 
If the harvesting effort is constant, after appropriate substitutions, Theorems~\ref{theorem_harv1}, \ref{theorem_harv2}, 
\ref{theorem1a} and \ref{theorem2a} would lead to the description of possible outcomes. We expect
that the two species would coexist for a harvesting effort of the efficient consumer in a certain interval. For smaller efforts, a less efficient consumer goes extinct, while for higher efforts it becomes the only survivor in the competition.

It would be interesting to extend the present analysis to the case when harvesting is semi-constant \cite{Roques2007},
being constant everywhere but for small population levels, where it is close to proportional to the population size.
Such an assumption allows us to avoid the situation when solutions become negative. Certainly, to apply the same results, we have to suppose that, though independent of the population size, the harvesting function is spatially heterogeneous,
being proportional to both the growth rate and the carrying capacity. The study of significant and remnant harvested populations (introduced in \cite{Roques2007}) will strongly depend on the bounds of the carrying capacity.

Note that the results of the present paper are based on the essential assumption that the carrying capacity is positive on the whole domain. However, both in literature and practical applications, the carrying capacity can become zero at some locations, especially close to the boundary of the domain. Zero or even negative growth rates can be considered, see, for example, \cite{Lam2015}. For the type of diffusion discussed in the present paper, the carrying capacity driven diffusion, existence of areas separated by the strips with the zero carrying capacities can lead to separated non-mixing patches, while
the other species would freely spread in the whole domain.
It would be interesting to explore how a complicated structure of the domain, in particular, quickly changing $r$ and $K$, 
can influence survival. In Example~\ref{ex2}, high variation of carrying capacities promotes survival of the species with a carrying capacity driven diffusion.

Unlike \cite{Korobenko2014}, in the present paper both species are subject to the logistic type of growth.  
However, the case when there is a strong Allee effect, 
i.e. there is a minimum population size guaranteeing survival, even without harvesting, is still an open problem.
The influence of the growth law on harvesting outcome is outlined in \cite{Korobenko2013}: 
higher per capita growth rates for smaller population densities are crucial. 
It would also be interesting to explore competition of  the two species with different growth laws.

\section{Auxiliary Results and Proofs}
\label{sec:proofs}

We start with a result on a stationary solution weighted average.

\begin{Lemma} \cite[Lemma 3]{CAMWA2016}
\label{Lsteady2}
If $v^{\ast}$ is a positive solution of (\ref{semi_v}) then
\begin{equation}\label{eq_aux_lem1}
\int \limits_\Omega r(x) K_2(x) \,dx > \int \limits_\Omega r(x) v^*(x) \, dx.
\end{equation}
\end{Lemma}

As a partial case of \eqref{eq:main_modified},  consider a competition model with proportional, 
space-dependent carrying capacities and growth rates with  $r_1>0$, $r_2>0$
\begin{equation}
\label{eq:modified}
\begin{aligned}
&\frac{\partial u(t,x)}{\partial t} =  \nabla \cdot \left[ a(x) \nabla \left( \frac{u(t,x)}{P(x)} \right) \right]  + r_1 r(x)u(t,x) \left(1-\frac{u(t,x)+v(t,x)}{cK(x)} \right), \\ 
&\frac{\partial v(t,x)}{\partial t} = \nabla \cdot \left[ b(x) \nabla \left( \frac{v(t,x)}{Q(x)} \right) \right] + r_2 r(x)v(t,x) \left(1-\frac{u(t,x)+v(t,x)}{K(x)} \right), \\
&t>0, \quad x \in \Omega, \\
&\frac{\displaystyle \partial u}{\displaystyle \partial n} - \frac{\displaystyle u}{\displaystyle P} \frac{\displaystyle 
\partial P}{\displaystyle \partial n}  
=\frac{\displaystyle \partial v}{\displaystyle \partial n}  - 
\frac{\displaystyle v}{\displaystyle Q} \frac{\displaystyle \partial Q}{\displaystyle \partial n} =0,~x\in 
{\partial}{\Omega},\\
& u(0,x)=u_0(x), \;v(0,x)=v_0(x),\;x\in{\Omega}.
\end{aligned}
\end{equation}

Since \eqref{eq:modified} is a monotone 
dynamical system \cite{CC}, as well as all the above systems, 
we will further use a modification of \cite[Theorem B, p. 4087]{Hsu} 
in the form earlier described in \cite[Theorem 17, pp.1203-1204]{Korobenko2014} in detail.   

\begin{Lemma}
\label{lemmaHsu}
Let $(0,0)$ be a repelling equilibrium of \eqref{eq:modified}.
Then exactly one of the following holds:
\begin{description}
\item[(a)] There exists a positive coexistence equilibrium of (\ref{eq:modified}).
\item[(b)] $(u,v)\rightarrow (u^{\ast},0)$ as $t\rightarrow\infty$ for  
$(u_{0},v_{0})$ from a prescribed set.
\item[(c)] $(u,v)\rightarrow (0,v^{\ast})$ as $t\rightarrow\infty$ for 
$(u_{0},v_{0})$ from a prescribed set.
\end{description}
Moreover, if (b) or (c) holds then for every non-trivial initial conditions 
either $(u,v)\rightarrow (u^{\ast},0)$ or $(u,v)\rightarrow (0,v^{\ast})$ as $t\rightarrow\infty$.
\end{Lemma}

\begin{Rem}
Note that by Lemma~\ref{lemmaHsu}, once the trivial equilibrium is a repeller, there is no coexistence 
equilibrium and one of semi-trivial equailibrium solutions is unstable, the other semi-trivial equilibrium is 
globally asymptotically stable.

If zero is a repeller and both semi-trivial equailibrium solutions are unstable, all solutions converge to a 
coexistence equilibrium.
\end{Rem}

We start with auxiliary statements used in the proof of Theorem~\ref{theorem_harv1},
in particular, we use special cases of the results in \cite{BKK2015}.

\begin{Lemma} \cite{BKK2015}
\label{lem1}
If $P$ is proportional to $K$, $\nabla \cdot [b\nabla (K/Q)] \not\equiv 0$ on $\Omega$ and $c\in [1,\infty)$ then the 
semi-trivial equilibrium $(cK,0)$ of system \eqref{eq:modified} is globally asymptotically stable.
\end{Lemma}

\begin{Lemma} \cite{BKK2015}
\label{lem2}
If $c\in (0,1)$ then the semi-trivial equilibrium $(cK,0)$ of \eqref{eq:modified} is unstable.
\end{Lemma}

\begin{Lemma}  \cite{BKK2015}
\label{lem3}
Let $K_1$ and $K_2$ be positive on $\overline{\Omega}$, $r$ be nonnegative in $\overline{\Omega}$ and positive
in an open nonempty subdomain of $\Omega$, $r_1>0$, $r_2>0$. Then the trivial equilibrium $(0,0)$
of \eqref{eq:main_modified} is a repeller.
\end{Lemma}

{\bf 
However, in the case $c\in (0,1)$, the paper \cite{BKK2015} does not give an answer whether the two species would eventually 
coexist, or only the second one survives while the first one goes extinct. 
In \cite{BKK2015}, numerical simulation illustrated that for $c<1$ close enough to one, there is coexistence, while for $c$ sufficiently small, the semi-trivial equilibrium $(0,v^*)$ is globally asymptotically stable. Here we prove existence of such $c^*$  that coexistence is inevitable for any $c \in (c^*,1)$.
}

\begin{Lemma}
\label{lem4}
Let $P$ be proportional to  $K$, while $\nabla \cdot \left[ b(x) \nabla \left( \frac{K(x)}{Q(x)} \right) \right] \not\equiv 0$ on $\Omega$. 
Then there exists $c_1 \in (0,1)$ such that for $c \in (c_1 ,1)$, the semi-trivial equilibrium $(0,v^*)$ of \eqref{eq:modified}, where $v^*$ satisfies \eqref{semi_v} with $K_2=K$, is unstable.
\end{Lemma}
\begin{proof}
For the linearization of the first equation in \eqref{eq:modified} around the equilibrium $(0,v^*)$
$$
\frac{\partial u(t,x)}{\partial t} = \nabla \cdot \left[ a(x) \nabla \left( \frac{u(t,x)}{P(x)} \right) \right]  + r_1 r(x)u(t,x) \left(1-\frac{v^*(x)}{cK(x)} \right),
$$
consider the associated eigenvalue problem 
\begin{equation}
\label{eig_p1}
\nabla \cdot \left[ a \nabla \left(\frac{\displaystyle \psi }{\displaystyle P}\right) \right]
+r_1 r \psi  \left(1-\frac{\displaystyle v^{\ast}}{\displaystyle c K(x)}\right)=\sigma \psi,\; x\in \Omega,\;
 \frac{\displaystyle \partial (\psi/P)}{\displaystyle \partial n}=0,\; x\in\partial\Omega.
\end{equation}
According to \cite{CC}, the principal eigenvalue $\sigma_1$ of (\ref{eig_p1}) can be computed as 
\begin{equation}
\label{eig_p1a}
\sigma_1 =
\sup_{\psi \neq 0, \psi\in W^{1,2}} \left. \left[
-\int \limits_\Omega a |\nabla (\psi/P)|^2\,dx
+\int \limits_\Omega r_1 r  \frac{\psi^2}{P} \left(1-\frac{\displaystyle v^{\ast}}{\displaystyle cK}\right)\,dx\right]
\right/
\int \limits_\Omega \frac{\psi^2}{P}\,dx.
\end{equation}
Let $c=1$, $\sigma_1$ be the principal eigenvalue, and  $\Psi$ be the principal eigenfunction, then 
$$
-\int \limits_\Omega a |\nabla (\Psi/P)|^2\,dx
+\int \limits_\Omega r_1 r \frac{\Psi^2}{P} \left(1-\frac{\displaystyle v^{\ast}}{\displaystyle K}\right)\,dx
-\sigma_1 \int \limits_\Omega \frac{\Psi^2}{P}\,dx = 0.
$$
We have 
$$0< \int \limits_\Omega \frac{\Psi^2}{P} r_1 r \frac{v^*}{K}\, dx = -\int \limits_\Omega a |\nabla (\Psi/P)|^2\,dx
+\int \limits_\Omega r_1 r \frac{\Psi^2}{P} \,dx
-\sigma_1 \int \limits_\Omega \frac{\Psi^2}{P}\,dx.
$$ 
Note that $v^*$ does not depend on $c$, and for 
$$
c> c_1 :=\int \limits_\Omega \frac{\Psi^2}{P} r_1 r \frac{v^*}{K}\, dx \left[ -\int \limits_\Omega a |\nabla (\Psi/P)|^2\,dx
+\int \limits_\Omega r_1 r(x) \frac{\Psi^2}{P} \,dx
\right]^{-1},
$$
the principal eigenvalue of \eqref{eig_p1} is positive, since the substitution of $\psi=\Psi$ in the brackets of \eqref{eig_p1a} gives a positive number. 

Further, we construct a simpler estimate of $c$ guaranteeing coexistence, not involving $\Psi$.

Choosing $\psi(x)= K(x)$, denoting
$\displaystyle
M:=\int \limits_\Omega  K(x) \,dx ,
$
and recalling that $P$ is proportional to $K$,
we observe by Lemma~\ref{Lsteady2} that the principal eigenvalue satisfies
$$
\sigma_1 \geq \frac{r_1}{M} \int \limits_\Omega r(x) K(x)\left(1-\frac{v^{\ast}(x)}{cK(x)}\right) 
\,dx = \frac{r_1}{M} \left[ \int \limits_\Omega r(x) K(x)\, dx - \frac{1}{c} \int \limits_\Omega r(x) v^*(x) \, dx \right]>0
$$
for $c\in (c^*,1)$, where
\begin{equation}
\label{eq_c_ast}
c^* = \frac{\int \limits_\Omega r(x) v^*(x) \, dx}{\int \limits_\Omega r(x) K(x)\, dx}<1.
\end{equation}
Therefore, for $c\in (c^*,1)$, $\sigma_1$ is positive, and thus
the semi-trivial steady state $(0,v^*)$ of (\ref{eq:modified}) is unstable.
\end{proof}

All the above results on local instability of semi-trivial equilibrium solutions and the fact that the trivial equilibrium is a repeller, by Lemma~\ref{lemmaHsu}, imply the following result. 

\begin{Prop}
\label{theorem1} 		
Let $P$ be proportional to  $K$, while $\displaystyle \nabla \cdot \left[ b(x) \nabla \left( \frac{K}{Q} \right) \right] \not\equiv 0$ on $\Omega$. 
{\bf
There exists $c_1 \in (0, 1)$ such that $c_1 \leq c^*$, where $c^*$ is defined in \eqref{eq_c_ast}, and
for $c \in (c_1,1)$ all solutions of
\eqref{eq:modified} strongly persist. 
}
\end{Prop}

\begin{Lemma}
\label{lem5}
Let $P$ be proportional to  $K_1=K$, let $\displaystyle \nabla \cdot \left[ b(x) \nabla \left( \frac{K}{Q} \right) \right] \not\equiv 0$ on $\Omega$. 
There exists $c_2 \in (0,1)$ such that whenever $c \in(0, c_2)$, system \eqref{eq:modified}
has no 
coexistence equilibrium.
\end{Lemma}
	
\begin{proof}
Everywhere we assumed $K(x)\in C^{1+p}(\overline \Omega)$, $K(x)>0$ for all $x\in \overline \Omega$. These assumptions on $K(x)$ imply existence of positive lower and upper bounds for $K$ 
\begin{equation}
\label{eq:K_bounds}
0<m:=\min_{x\in\overline \Omega }K(x),\, M:=\max_{x\in\overline \Omega }K(x).
\end{equation}
By \cite[Appendix Proposition 4]{Korobenko2013} and \cite[Propositions 3.2 and 3.3]{CC}, there exist similar positive bounds for a positive stationary solution $v^*(x)$ of \eqref{semi_v}
\begin{equation}
\label{eq:v_star_bounds}
0<b:=\min_{x\in\overline \Omega }v^*(x),\, B:= \max_{x\in\overline \Omega }v^*(x).
\end{equation}
We have $b \leq M$; assuming the contrary, we obtain that $v^*(x)>K(x)$ for any $x$, thus integration of \eqref{semi_v} with $K_2 \equiv K$ over $\Omega$ leads to a contradiction that the integral of a negative on $\Omega$ function equals zero. 	Thus $b/(2M)<1$; 	
let us choose $c \in\left(0,\frac{b}{2M} \right)$. Let $(u_s,v_s)$ be a coexistence solution of \eqref{eq:modified}. 
Substituting $\overline{u} = \frac{bK(x)}{2M}$ into the first equation in \eqref{eq:modified}, assuming $v_s\geq 0$ and taking into account that the diffusion term vanishes
as $\overline{u}/P$, or $\overline{u}/K$, is constant, we obtain
$$
r(x)\overline{u} \left(1- \frac{\overline{u}+v_s(x)}{cK(x)} \right)\leq  r(x)\overline{u} \left(1- \frac{\overline{u}}{cK(x)} \right) =  r(x)\overline{u} \left( 1- \frac{b}{2cM}\right)<0
$$
for any $x \in \Omega$ since $c< \frac{b}{2M}$.
Thus $\overline{u} = \frac{bK}{2M}$ is an upper solution of the equation, see \cite[Chapter 8, Definition 1.2]{Pao}. 
We can choose zero as a lower solution and immediately get by \eqref{eq:K_bounds}, \eqref{eq:v_star_bounds},
\begin{equation}
\label{eq:u_s_bounds}
0\leq u_s(x)\leq  \frac{bK(x)}{2M} 
\leq \frac{b}{2} \leq \frac{v^*(x)}{2}.
\end{equation}
Choosing $\underline{v} = \frac{v^*}{2}$ in the second equation of \eqref{eq:modified}, we can see that it is a lower solution (see \cite[Chapter 8, Definition 1.2]{Pao}), since $v^*$ is a solution of \eqref{semi_v} with $K_2 \equiv K$:
$$
\nabla \cdot \left[ b(x) \nabla \left( \frac{\underline{v}}{Q} \right) \right] + 	
r_2 r\underline{v}  \left(1-\frac{u_s+\underline{v} }{K} \right) \geq 
\frac{1}{2} \nabla \cdot \left[ b(x) \nabla \left( \frac{v^*}{Q} \right) \right]
+ r_2 r\underline{v} \left(1-\frac{0.5 v^*+ 0.5v^*}{K} \right) = 0.
$$
Therefore, $v_s\geq \frac{v^*}{2}$. Finally, using this inequality together with the fact that $u_s\geq 0$ in the first equation of \eqref{eq:modified}, we get, due to $c<\frac{B}{2M}$,
\begin{align*}
0= &\nabla \cdot \left[ a(x) \nabla \left( \frac{u_s(x)}{P(x)} \right) \right] +  
r_1r(x)u_s(x) \left(1-\frac{u_s(x)+v_s(x)}{cK(x)} \right)
\\ \leq &\nabla \cdot \left[ a(x) \nabla \left( \frac{u_s(x)}{P(x)} \right) \right] + r_1 r(x)u_s(x)  
\left(1-\frac{v_s(x)}{cK(x)} \right) 
\\ \leq &\nabla \cdot \left[ a(x) \nabla \left( \frac{u_s(x)}{P(x)} \right) \right] + 
r_1 r(x)u_s(x) \left(1 -\frac{v^*(x)}{2c K(x)} \right) \\ \leq 
& \nabla \cdot \left[ a(x) \nabla \left( \frac{u_s(x)}{P(x)} \right) \right]  + r_1 r(x)u_s(x) \left( 1 -\frac{b}{2cM} \right) \leq \left[ a(x) \nabla \left( \frac{u_s(x)}{P(x)} \right) \right],
		\end{align*}
where the strict inequality is valid in the first and the last cases, unless $u_s(x)\equiv 0$. Integrating both sides of the inequality above, using the boundary conditions and assuming $u_s \not\equiv 0$ leads to
		\begin{align*}
		0 = & \int_{\Omega} \left( \nabla \cdot \left[ a(x) \nabla \left( \frac{u_s(x)}{P(x)} \right) \right] \right)dx + \int_{\Omega} r_1 r(x)u_s(x) \left(1-\frac{u_s(x)+v_s(x)}{cK(x)} \right)dx \\  < & \int_{\Omega}  \left( \nabla \cdot \left[ a(x) \nabla \left( \frac{u_s(x)}{P(x)} \right) \right] \right)dx =0,
		\end{align*}
		which is a contradiction. Hence, there is no coexistence equilibrium.
	\end{proof}

Now, Lemma \ref{lemmaHsu} leads to the global stability of $(0,v^*)$, as stated below.

\begin{Prop}
\label{theorem2} 		
Let $P$ be proportional to  $K$, while $\displaystyle \nabla \cdot \left[ b(x) \nabla \left( \frac{K}{Q} \right) \right] \not\equiv 0$ on $\Omega$. 
There exists $c_2 \in (0, 1)$ such that for $c \in (0,c_2)$ all solutions of
\eqref{eq:modified} converge to the semi-trivial equilibrium $(0,v^*)$.
\end{Prop}

Now, we can proceed to the proof of the main result.


\begin{proof} ({\bf the proof of Theorem~\ref{theorem_harv1}})
Let $\beta \in [0,1)$. Then \eqref{eq:main_problem} becomes \eqref{eq:modified} with
\begin{equation}
\label{transition}
K \to (1-\beta)K, ~~ c=\frac{1-\alpha}{1-\beta}, ~~r_1= 1-\alpha,~~r_2=1-\beta.
\end{equation}

1) If $\alpha \leq \beta$, we immediately get $c \geq 1$. By Lemma~\ref{lem1},
the semi-trivial equilibrium $(u^{\ast},0)=((1-\alpha)K,0)$ is globally asymptotically stable.

2) Fix $\beta \in [0,1)$. 
{\bf By Proposition~\ref{theorem1},
there is $\displaystyle c_1=\frac{1-\alpha_1}{1-\beta}\in (0,1)$ such that for any $c \in (c_1,1)$, all solutions 
of \eqref{eq:modified} strongly persist},
where $c_1 \leq c^{\ast}$, $c^{\ast}$ is defined in 
\eqref{eq_c_ast}. Since $\alpha_1=1-c_1(1-\beta)$, taking into account that $K$ is substituted by $(1-\beta)K$ and $v^{\ast}=v_{\beta}^{\ast}$
is also $\beta$-dependent, the bound $\alpha^{\ast}=1-c^{\ast}(1-\beta)$ has form \eqref{alpha_star}. For $\alpha \in (\beta, \alpha_1)$,
which corresponds to $c \in (c_1,1)$, by Proposition~\ref{theorem1}, {\bf all solutions strongly persist.}

3) Again, for a fixed $\beta \in [0,1)$, system \eqref{eq:main_problem} becomes \eqref{eq:modified} with 
\eqref{transition}. By Proposition~\ref{theorem2}, there exists a $\displaystyle c_2=\frac{1-\alpha_2}{1-\beta}\in (0,1)$
such that for any $c \in (0,c_2)$, all solutions
of \eqref{eq:modified} and thus all solutions of \eqref{eq:main_problem} converge to $(0,v_{\beta}^{\ast})$. Choosing 
$\displaystyle \alpha_2=1-c_2(1-\beta) \in (0,1)$, we get that for $\alpha \in (\alpha_2,1)$, all solutions of 
\eqref{eq:main_problem} converge to $(0,v_{\beta}^{\ast})$, which concludes the proof.
\end{proof}

\begin{Rem}
If we consider $\beta \geq 1$ then, from the second equation in \eqref{eq:main_problem}, 
$\displaystyle \lim_{t \to \infty} v(t)=0$, independently of  $u(t) \geq 0$. The first equation implies,
for any $\alpha \in [0,1)$, $\displaystyle \lim_{t \to \infty} u(t)=(1-\alpha)K$, thus the semi-trivial 
equilibrium $((1-\alpha)K,0)$  is globally asymptotically stable.
\end{Rem}


Next, we proceed to the case of the {\em ideal free pair} when \eqref{no_K} is satisfied and
for some $\gamma>0$ and $\delta>0$, identity \eqref{hull} holds.

Let us present auxiliary statements used in the proof of Theorem~\ref{theorem_harv2}.
We further use the following lemma obtained in \cite{Lou} and recently discussed in \cite{Ang2016}.

\begin{Lemma} \cite{Lou}
\label{lem6}
If $u^*/K_1 \not\equiv 1$ then the solution $u^*$ of \eqref{semi_u} satisfies 
\begin{equation}
\label{Lou_u}  
\int_{\Omega} r P \left( \frac{u^*}{K} -1 \right) >0.
\end{equation}
Similarly, if $v^*/K_2 \not\equiv 1$ then for the solution $v^*$ of \eqref{semi_v} the following inequality holds 
\begin{align*}
\int_{\Omega} r Q \left( \frac{v^*}{K} -1 \right) >0.
\end{align*}
\end{Lemma}

\begin{Rem}
\label{remark1}
If in \eqref{semi_u} we assume constant $a$ and $P$ and $r=K$, we get the stationary solution of the Neumann problem for Fisher's equation
\begin{equation*}
d \Delta u^{\ast} + 
u^* (x) \left( K(x)-u^*(x) \right)=0,\;x\in\Omega,~~~
\frac{\partial u^*}{\partial n}=0,\;x\in\partial\Omega,
\label{semi_u_standard}
\end{equation*}
for which \eqref{Lou_u} leads to an important conclusion
\begin{equation}
\label{higher_average}
\int_{\Omega} u^*(x)~dx > \int_{\Omega} K(x)~dx
\end{equation}
that the average solution levels are higher than that of the carrying capacity, while \eqref{eq_aux_lem1} leads to the opposite result for space-independent $r$.
\end{Rem}

\begin{Lemma} 
\label{lem7}
Let \eqref{no_K} and \eqref{hull} hold for some $\gamma>0$, $\delta>0$, and $c \in (0,1]$. Then the semi-trivial equilibrium $(u^*,0)$ of \eqref{eq:modified} is unstable.
\end{Lemma}

\begin{proof}
Integrating \eqref{semi_u} with $K_1=cK$ and using the boundary conditions for $u^*$, we get
\begin{align*}
0 = & \int_{\Omega} r_1 r u^* \left( 1 - \frac{u^*}{cK} \right)~dx 
\\ = & ~r_1  \int_{\Omega} \frac{r}{cK} (u^* - cK) \left( cK - u^* \right)~dx + c r_1 \int_{\Omega} r K \left( 1 - \frac{u^*}{cK} \right)~dx 
\\ = & - r_1  \int_{\Omega} \frac{r}{cK} (u^* - cK)^2 ~dx +c \gamma r_1 \int_{\Omega} r P \left( 1 - \frac{u^*}{cK} \right)~dx 
+ c \delta r_1 \int_{\Omega} r Q \left( 1 - \frac{u^*}{cK} \right)~dx \\
 =  & - r_1  \int_{\Omega} \frac{r}{cK} (u^* - cK)^2 ~dx  - c \gamma r_1 \int_{\Omega} r P \left( \frac{u^*}{cK} - 1 \right)~dx
+ c \delta r_1 \int_{\Omega} r Q \left( 1 - \frac{u^*}{cK} \right)~dx.
\end{align*}
Here a substitution of $u^*$ in \eqref{semi_u} with $K_1=cK$, where \eqref{no_K} holds, leads to the conclusion that $u^*-cK 
\not\equiv 0$. Thus the first term in the right-hand side above is negative. By \eqref{Lou_u} in
Lemma~\ref{lem6}, the second term is also negative. Since the sum equals to zero,
$$
0 < \int_{\Omega} r Q \left( 1 - \frac{u^*}{cK} \right)~dx.
$$
Moreover, as  $c \in (0,1]$ and
$$
\int_{\Omega} r Q \left( 1 - \frac{u^*}{cK} \right)~dx = \int_{\Omega} r Q \left( 1 - \frac{u^*}{K} \right)~dx - \frac{1-c}{c} 
\int_{\Omega} r Q \frac{u^*}{K}~dx >0,  
$$
with $(1-c)/c \geq 0$, the first term in the right-hand side of the equality above is positive
\begin{equation}
\label{22}
\int_{\Omega} r Q \left( 1 - \frac{u^*}{K} \right)~dx >0.
\end{equation}

Next, consider the  eigenvalue problem for the linearization of the second equation 
in \eqref{eq:modified} around the equilibrium $(u^*,0)$
\begin{equation}\label{eig_p2}
\nabla \cdot \left[ b \nabla \left(\frac{\displaystyle \psi }{\displaystyle Q}\right) \right]
+r_2 r \psi  \left(1-\frac{\displaystyle u^{\ast}}{\displaystyle K(x)}\right)=\sigma \psi,\; x\in \Omega,\;
 \frac{\displaystyle \partial (\psi/Q)}{\displaystyle \partial n}=0,\; x\in\partial\Omega.
\end{equation}
According to \cite{CC}, the principal eigenvalue $\sigma_1$ of (\ref{eig_p2}) can be computed as 
\begin{align*}
\sigma_1 =
\sup_{\psi \neq 0, \psi\in W^{1,2}} \left. \left[
-\int \limits_\Omega b |\nabla (\psi/Q)|^2\,dx
+\int \limits_\Omega r_2\,r \frac{\psi^2}{Q} \left(1-\frac{\displaystyle u^{\ast}}{\displaystyle K}\right)\,dx\right]
\right/
\int \limits_\Omega \frac{\psi^2}{Q}\,dx.
\end{align*}
It is  not less than the value computed for $\psi=Q$ which is
$$
\sigma_1 \geq \left. r_2 \int\limits_{\Omega} r Q \left( 1 - \frac{u^*}{K} \right)~dx \right/ \int\limits_\Omega  Q\,dx >0
$$ 
by \eqref{22}. Since the principal eigenvalue is positive, the equilibrium $(u^*,0)$ is unstable.
\end{proof}

\begin{Lemma} 
\label{lem8}
Let \eqref{no_K} and \eqref{hull} hold for some $\gamma>0$, $\delta>0$. Then there exists  $c_1 \in (0,1)$
such that for $c \in (c_1,1)$, the semi-trivial equilibrium $(0,v^*)$ of \eqref{eq:modified} is unstable.
\end{Lemma}

\begin{proof}
Integrating \eqref{semi_v} over $\Omega$, taking into account the boundary conditions, gives
\begin{align*}
0  = & \int_{\Omega} r_2 r v^* \left( 1 - \frac{v^*}{K} \right)~dx 
\\  = & ~ r_2  \int_{\Omega} \frac{r}{K} (v^* - K) \left( K - v^* \right)~dx 
+  r_2 \int_{\Omega} r K \left( 1 - \frac{v^*}{K} \right)~dx 
\\  = & - r_2  \int_{\Omega} \frac{r}{K} (v^* - K)^2 ~dx + \gamma r_2 \int_{\Omega} r P \left( 1 - \frac{v^*}{K} \right)~dx 
+ \delta r_2 \int_{\Omega} r Q \left( 1 - \frac{v^*}{K} \right)~dx \\
 = & - r_2  \int_{\Omega} \frac{r}{K} (v^* - K)^2 ~dx  +  \gamma r_2 \int_{\Omega} r P \left( 1- \frac{v^*}{K}  \right)~dx
- \delta r_2 \int_{\Omega} r Q \left( \frac{v^*}{K} -1 \right)~dx, 
\end{align*}
where the first and the third terms are negative. 

Thus the second term is positive, and
$\displaystyle
\int_{\Omega} r P \left( 1 - \frac{v^*}{K} \right)~dx >0.
$
Moreover, denoting 
\begin{equation}
\label{c_ast}
c^* = \frac{ \int_{\Omega} P r v^*/K~dx}{\int_{\Omega}  r P~dx} <1,
\end{equation}
we obtain that
\begin{equation}
\label{22a}
\int_{\Omega} r P \left( 1 - \frac{v^*}{cK} \right)~dx >0, \quad c \in (c^*,1].
\end{equation}

Consider eigenvalue problem \eqref{eig_p1} for the linearization of the first equation in \eqref{eq:modified}.
Its principal eigenvalue $\sigma_1$ is given by \eqref{eig_p1a}, and it is not less than the value computed for $\psi=P$, i.e.
$$
\sigma_1 \geq \left. r_1 \int\limits_{\Omega} r P \left( 1 - \frac{v^*}{cK} \right)~dx \right/ \int\limits_\Omega  P\,dx >0
$$ 
for $c \in (c_1,1]$, and $c_1 \leq c^*$,
by \eqref{22a}, where $c^*$ is defined in \eqref{c_ast}. 
Since the principal eigenvalue is positive, the equilibrium $(0,v^*)$ for $c\in (c^*,1)$ is unstable.
\end{proof}

\begin{Lemma} 
\label{lem9}
Let \eqref{no_K} and \eqref{hull} hold for some $\gamma>0$, $\delta>0$. 
There exists $c_2 \in (0,1)$ such that whenever $c \in(0, c_2)$, system \eqref{eq:modified}
has no coexistence equilibrium.
\end{Lemma}
	
\begin{proof}
As we assumed $K,P \in C^{1+p}(\overline \Omega)$, $K(x)>0$, $P(x)>0$ for all $x\in \overline \Omega$, there is an upper bound for $K$ in  \eqref{eq:K_bounds},  $M:=\max_{x\in\overline \Omega }K(x)$.
Similar smoothness estimates are valid for $P(x)$ and the ratio $K/P$, and the bounds are positive
$$  
0<g :=\min_{x\in\overline \Omega } \frac{K(x)}{P(x)},\, G:= \max_{x\in\overline \Omega } \frac{K(x)}{P(x)}.
$$
Also, as mentioned in the proof of Lemma~\ref{lem5}, the lower bound for a positive stationary solution $v^*(x)$ of \eqref{semi_v} with $K_2=K$
is positive $0<b:=\min_{x\in\overline \Omega }v^*(x)$, see \eqref{eq:v_star_bounds}. Again, $b \leq M$, $b/(2M)<1$ and, by the definition of $g$ and $G$,
$\displaystyle g \leq \frac{K(x)}{P(x)} \leq G$.
We choose
$$c^{**} := \frac{bg}{2GM} \leq \frac{1}{2},\quad c \in (0,c^{**}).
$$
Then $\displaystyle c<\frac{b}{2M}$.
Let $\displaystyle \overline{u}(x)  = \frac{b g P(x)}{2M}$.
Substituting $\overline{u}$ into the first equation in \eqref{eq:modified}, assuming $v_s\geq 0$ and noticing that the diffusion term vanishes
as $\overline{u}/P$  is constant, we obtain
$$
r \overline{u} \left(1- \frac{\overline{u}+v_s}{cK} \right)\leq  r\overline{u} \left(1- \frac{\overline{u}}{cK} \right) 
=  r\overline{u} \left( 1- \frac{b g P(x)}{2cKM}\right)
\leq r\overline{u} \left(  1 - \frac{bg}{2cGM} \right) <0
$$
for any $x \in \Omega$.  
Thus $\overline{u} = \frac{bgP(x)}{2M}$ is an upper solution, choosing zero as a lower solution, we get
$$
0\leq u_s(x)\leq\frac{bgP(x)}{2M} \leq \frac{b}{2} \leq \frac{v^*(x)}{2},
$$
as in \eqref{eq:u_s_bounds}.
Choosing $\underline{v} = \frac{v^*}{2}$ in the second equation of \eqref{eq:modified}, we can see that it is a lower solution, see \cite[Chapter 8, Definition 1.2]{Pao}, since $v^*$ is a solution of \eqref{semi_v} with $K_2 \equiv K$:
\begin{align*}
& \nabla \cdot \left[ b(x) \nabla \left( \frac{\underline{v}}{Q} \right) \right] + 	
r_2 r(x)\underline{v}  \left(1-\frac{u_s+\underline{v} }{K} \right) \\ \geq & 
\frac{1}{2} \nabla \cdot \left[ b(x) \nabla \left( \frac{v^*}{Q} \right) \right]
+ r_2 r(x)\underline{v} \left(1-\frac{0.5 v^*+ 0.5v^*}{K} \right) = 0.
\end{align*}
Thus, $v_s\geq \frac{v^*}{2}$. Finally, using this inequality together with the fact that $u_s\geq 0$ in the first equation of \eqref{eq:modified}, we get
		\begin{align*}
0 = & \nabla \cdot \left[ a(x) \nabla \left( \frac{u_s(x)}{P(x)} \right) \right] +  r_1 r(x)u_s(x) 
\left(1-\frac{u_s(x) + v_s(x)}{cK(x)} \right)
\\  \leq & \nabla \cdot \left[ a(x) \nabla \left( \frac{u_s(x)}{P(x)} \right) \right] + r_1 r(x)u_s(x)  
\left(1-\frac{v_s(x)}{cK(x)} \right) 
\\ \leq  & \nabla \cdot \left[ a(x) \nabla \left( \frac{u_s(x)}{P(x)} \right) \right] + 
r_1 r(x)u_s(x) \left(1 -\frac{v^*(x)}{2c K(x)} \right) \\ \leq &
 \nabla \cdot \left[ a(x) \nabla \left( \frac{u_s(x)}{P(x)} \right) \right]  + r_1 r(x)u_s(x) \left( 1 -\frac{b}{2cM} \right)\\
 \leq & \nabla \cdot \left[ a(x) \nabla \left( \frac{u_s(x)}{P(x)} \right) \right],
		\end{align*}
and the strict inequality is valid unless $u_s \equiv 0$. Integrating both sides of the inequality 
above using boundary conditions and assuming $u_s \not\equiv 0$ leads to
$$
0 =  \int_{\Omega}  \nabla \cdot \left[ a \nabla \left( \frac{u_s}{P} \right) \right] \, dx + \int_{\Omega}r(x)u_s 
\left(1-\frac{u_s+v_s}{cK} \right)dx <  \int_{\Omega}  \nabla \cdot \left[ a\nabla \left( \frac{u_s}{P} \right) \right] \, dx =0.
$$
The contradiction proves that  there is no coexistence equilibrium.
\end{proof}

Similarly to Lemma~\ref{lem8}, we  establish stability for $c \in [1,\infty)$.

\begin{Lemma} 
\label{lem8a}
Let \eqref{no_K} and \eqref{hull} hold for some $\gamma>0$, $\delta>0$. Then there exists  $c_1 \in (1,\infty)$
such that for $c \in [1,c_1)$, the semi-trivial equilibrium $(u^*,0)$ of \eqref{eq:modified} is unstable.
\end{Lemma}

By Lemma~\ref{lem8}, for certain $c$, 
$(0,v^*)$ is unstable, while by Lemma~\ref{lem8a}, for another set of $c$, $(u^*,0)$ is unstable.
Whenever $c$ satisfies both conditions, Lemma~\ref{lemmaHsu} implies coexistence conditions. 

\begin{Prop}
\label{theorem1a} 
Let \eqref{no_K} and \eqref{hull} hold for some $\gamma>0$, $\delta>0$. 
{\bf 
There exist $c_1 \in (0,1)$ and $c_2 \in (1,\infty)$ such that whenever $c \in(c_1,c_2)$, 
all solutions of \eqref{eq:modified} strongly persist.
}
\end{Prop}

Next, competitive exclusion can be considered for $c>1$ as well.

\begin{Lemma} 
\label{lem9a}
Let \eqref{no_K} and \eqref{hull} hold for some $\gamma>0$, $\delta>0$. 
There exists $c_3 \in (1,\infty)$ such that whenever $c \in(c_3,\infty)$, system \eqref{eq:modified}
has no coexistence equilibrium.
\end{Lemma}

Based on Lemmata~\ref{lem9} and \ref{lem9a}, we establish stability of the semi-trivial equilibria.

\begin{Prop}
\label{theorem2a} 		
Let \eqref{no_K} and \eqref{hull} hold for some $\gamma>0$, $\delta>0$. 
There exist $c_1 \in (0, 1)$ and $c_2 \in(1,\infty)$ such that for $c \in (0,c_1)$ all solutions of
\eqref{eq:modified} converge to the semi-trivial equilibrium $(0,v^*)$, while for $c \in (c_2,\infty)$
all solutions of \eqref{eq:modified} converge to the semi-trivial equilibrium $(u^*,0)$.
\end{Prop}

At this stage, we can proceed to the proof of Theorem~\ref{theorem_harv2}.

\begin{proof} ({\bf the proof of Theorem~\ref{theorem_harv2}})
We recall that  \eqref{eq:main_problem} has form \eqref{eq:modified} with notations
\eqref{transition}.

1) Fix $\beta \in (0,1)$. {\bf By Proposition~\ref{theorem1a}, there are $c_1\in(0,1)$ and $c_2\in(1,\infty)$
such that for $c\in (c_1,c_2)$, all solutions strongly persist.}
Denote $\alpha_1=1-\max\{1-c_2(1-\beta),0\}$,
$\alpha_2=1-c_1(1-\beta)$. Then $0\leq \alpha_1<\beta<\alpha_2<1$ and, by Proposition~\ref{theorem1a}, for $\alpha \in (\alpha_1,\alpha_2)$,
{\bf all solutions of \eqref{eq:main_problem} strongly persist.} 

2) For any fixed $\beta \in [0,1)$, by Proposition~\ref{theorem2a}, there is a $c_1\in (0,1)$ such that for $c\in (0,c_1)$, all solutions of 
\eqref{eq:modified} converge to $(0,v_{\beta}^*)$. Denote $\alpha_3=1-c_1(1-\beta) \in (\beta,1)$. 
Then for $\alpha \in (\alpha_3,1)$ we get $c \in (0,c_1)$ and thus all solutions of \eqref{eq:main_problem} converge to $(0,v_{\beta}^*)$,
where $v_{\beta}^*$ is a solution of \eqref{semi_v} with $K_2=(1-\beta)K$, $r_2=1-\beta$.

3) Denote $\displaystyle \beta_1=1- \frac{1}{c_2}$, where $c_2>1$ is defined in Proposition~\ref{theorem2a}, 
such that for all $c\in (c_2,\infty)$, all solutions of
\eqref{eq:modified} converge to $(u_{\alpha}^*,0)$. Next, let $\beta \in (\beta_1,1)$ be fixed, and $\alpha_0=1-(1-\beta)c_2$.
Then $c_2>1$ and $c \in (c_2,\infty)$ 
correspond to $0<\alpha_0<\beta<1$ and $\alpha \in [0,\beta)$, respectively.
Thus for any $\alpha\in [0,\alpha_0)$, all solutions of 
\eqref{eq:main_problem} converge to $(u_{\alpha}^*,0)$, where
$u_{\alpha}^*$ satisfies \eqref{semi_u} with $K_1=(1-\alpha)K$, $r_1=1-\alpha$.
This concludes the proof.
\end{proof}

Theorem~\ref{th_extinction} can be justified, similarly to \cite{Korobenko2013}, using differential inequalities.

\begin{proof} ({\bf the proof of Theorem~\ref{th_MSY}})
We recall that for a fixed $K$, the maximum of $f(x)=x(1-x/K)$ is attained at $x=K/2$ and equals $\frac{1}{4}K$.
Integrating the first equation in \eqref{eq:main_problem} over $\Omega$ and taking into account the boundary conditions, we get
\begin{align*}
{\rm SY} = & \int_{\Omega} \alpha r(x) u_s(x)~dx = \int_{\Omega}  r(x) u_s(x) \left( 1- \frac{u_s(x)+v_s(x)}{K(s)} \right) ~dx \\
\leq  & \int_{\Omega} r(x) u_s(x) \left( 1- \frac{u_s(x)}{K(x)} \right)~dx \leq \int_{\Omega} \frac{1}{4} r(x) K(x)~dx.
\end{align*}
For either $P=K$, $\alpha=0.5$, $\beta> 0.5$ and any $Q$, or $\beta \geq 0.5$ and $\displaystyle \nabla \cdot \left[ b(x) \nabla \left( \frac{K}{Q} \right) \right] \not\equiv 0$, the stationary solution $(0.5K(x),0)$
leads to \eqref{MSY}. 

Further, in the case of the ideal free pair, MSY in \eqref{MSY} is attained for appropriate $P$ and $Q$, $P+Q=K$, $\alpha=\beta=0.5$.
Similarly to the above case
\begin{align*}
{\rm SY} = & \int_{\Omega} \alpha r(x) u_s(x)~dx + \int_{\Omega} \beta r(x) v_s(x)~dx
\\ 
= & \int_{\Omega}  r(x) u_s(x) \left( 1- \frac{u_s(x)+v_s(x)}{K(s)} \right) ~dx+ 
\int_{\Omega}  r(x) v_s(x) \left( 1- \frac{u_s(x)+v_s(x)}{K(s)} \right)~dx
\\
= & \int_{\Omega}  r(x) (u_s(x)+v_s(x)) \left( 1- \frac{u_s(x)+v_s(x)}{K(s)} \right)~dx
\\ \leq & \int_{\Omega} \frac{1}{4} r(x) K(x)~dx,
\end{align*}
which concludes the proof.
\end{proof}

\begin{Rem}
\label{remark2}
Let us note that the maximum $\displaystyle \int_{\Omega} u(x)~dx$ may be attained for $\displaystyle \nabla \cdot \left[ a(x) \nabla \left( \frac{K}{P} \right) \right] \not\equiv 0$, see \eqref{higher_average}
in Remark~\ref{remark1}.
\end{Rem}

\section*{Acknowledgment}

{\bf
The authors are grateful to the referees whose valuable comments significantly contributed to the 
paper presentation. 
}
The authors were partially supported by the NSERC research grant RGPIN-2015-05976. {\bf The second 
author was also supported by Pacific Institute for the Mathematical Sciences with PIMS Student
Training Acceleration Award.
}


\end{document}